\documentclass{amsart}
\usepackage{amsmath, amssymb, amscd, amsthm, amsfonts, amstext, amsbsy,  mathrsfs, hyperref,  upgreek, mathtools, stmaryrd, enumitem, tensor, comment, lineno, tikz-cd}

\hypersetup{colorlinks=true}
\usepackage[textsize=footnotesize,color=blue!40
, bordercolor=blue!80]{todonotes}
\numberwithin{equation}{section}

\newcommand{\thistheoremname}{}
\newtheorem*{genericthm}{\thistheoremname}
\newenvironment{namedthm}[1]
{\renewcommand{\thistheoremname}{#1}
	\begin{genericthm}}
	{\end{genericthm}}

\def\hyp{{\hbox{-}}}

\newcommand{\zerohandgrenade}{0^{\: \mid^{\! \! \! \bullet}}}

\newtheorem{theorem}{Theorem}[section]
\newtheorem{lemma}[theorem]{Lemma}
\newtheorem{corollary}[theorem]{Corollary}
\newtheorem{proposition}[theorem]{Proposition}

\newtheorem*{question*}{Question}
\newtheorem*{proviso}{Proviso}

\theoremstyle{definition}

\newtheorem*{claim*}{Claim}
\newtheorem{definition}[theorem]{Definition}

\theoremstyle{remark}
\newtheorem{remark}[theorem]{Remark}

\newenvironment{enumerate-(a)}{\begin{enumerate}[label={\upshape (\alph*)}, leftmargin=2pc]}{\end{enumerate}}

\newenvironment{enumerate-(a)-r}{\begin{enumerate}[label={\upshape (\alph*)}, leftmargin=2pc,resume]}{\end{enumerate}}

\newenvironment{enumerate-(A)}{\begin{enumerate}[label={\upshape (\Alph*)}, leftmargin=2pc]}{\end{enumerate}}

\newenvironment{enumerate-(A)-r}{\begin{enumerate}[label={\upshape (\Alph*)}, leftmargin=2pc,resume]}{\end{enumerate}}

\newenvironment{enumerate-(i)}{\begin{enumerate}[label={\upshape (\roman*)}, leftmargin=2pc]}{\end{enumerate}}

\newenvironment{enumerate-(i)-r}{\begin{enumerate}[label={\upshape (\roman*)}, leftmargin=2pc,resume]}{\end{enumerate}}

\newenvironment{enumerate-(I)}{\begin{enumerate}[label={\upshape (\Roman*)}, leftmargin=2pc]}{\end{enumerate}}

\newenvironment{enumerate-(I)-r}{\begin{enumerate}[label={\upshape (\Roman*)}, leftmargin=2pc,resume]}{\end{enumerate}}

\newenvironment{enumerate-(1)}{\begin{enumerate}[label={\upshape (\arabic*)}, leftmargin=2pc]}{\end{enumerate}}

\newenvironment{enumerate-(1)-r}{\begin{enumerate}[label={\upshape (\arabic*)}, leftmargin=2pc,resume]}{\end{enumerate}}

\begin{document}

	\title{The $\lambda$-$\mathsf{PSP}$ at $\lambda$-$\Pi^1_1$ sets}
	
	\author{Fernando Barrera}
	\address[Fernando Barrera]
	{Institut f\"ur Diskrete Mathematik und Geometrie, Technische Universit\"at Wien, Wiedner Hauptstra{\ss}e 8-10/104, 1040 Vienna, Austria}
	\email[Fernando Barrera]{f.barrera.esteban@gmail.com}
	
	\author{Vincenzo Dimonte}
	\address[Vincenzo Dimonte]
	{GNSAGA INdAM and Universit\`a degli Studi di Udine, Via delle Scienze 206, 33100 Udine, Italy}
	\email[Vincenzo Dimonte]{vincenzo.dimonte@uniud.it}
	
	\author{Sandra Müller}
	\address[Sandra Müller]
	{Institut f\"ur Diskrete Mathematik und Geometrie, Technische Universit\"at Wien, Wiedner Hauptstra{\ss}e 8-10/104, 1040 Vienna, Austria}
	\email[Sandra Müller]{sandra.mueller@tuwien.ac.at}

    \begin{abstract}
	Given a strong limit cardinal $\lambda$ of countable cofinality, we show that if every (boldface) $\lambda\hyp\boldsymbol{\Pi}^1_1$ subset of the generalised Cantor space ${}^{\lambda}2$ has the $\lambda$-$\mathsf{PSP}$, then $0^\dagger$ exists. We show too that if every (lightface) $\lambda\hyp\Pi^1_1$ subset of ${}^\lambda 2$ has the $\lambda\hyp\mathsf{PSP}$, then there is an inner model with a measurable cardinal. The paper, a contribution to the ongoing research on generalised regularity properties in generalised descriptive set theory at singular cardinals of countable cofinality, is aimed at descriptive set theorists, and so it presents its results in as much detail as possible, particularly regarding the inner model-theoretic aspects. In doing so, we intend to provide the community with the tools needed to handle consistency strength arguments at the corresponding levels.
\end{abstract}

    \maketitle

    \section{Introduction}

    In recent years, the study of higher Baire and Cantor spaces has been subject of growing interest, although most of the literature has focused on versions of those spaces in which, roughly speaking, the role of $\omega$ is played by an uncountable regular cardinal. In this paper, however, singular cardinals of countable cofinality take on that role. Mainly motivated by the work of Cramer, Shi and Woodin on the set-theoretic descriptive properties of $L(V_{\lambda+1})$ under the very large cardinal assumption $I0(\lambda)$, properties similar to those that $L(\mathbb{R})$ exhibits under $\mathsf{AD}^{L(\mathbb{R})}$, the systematic study of generalised descriptive set theory ($\mathsf{GDST}$, from now on) at singular cardinals of countable cofinality has been initially carried out by Dimonte and Motto Ros (\cite{DMR}), and has been lately enriched by works of Dimonte, Lücke, Iannella, Poveda and Thei (\cite{dimonte2023descriptive},\cite{dimonte2024baire}). 
    
    As in classical descriptive set theory, one of the major themes in $\mathsf{GDST}$ is the study of the so-called (generalised) regularity properties, to which this paper contributes by providing a consistency strength lower bound to the satisfaction of the $\lambda$-Perfect Set Property at the lowest level in the $\lambda$-analytical and $\lambda$-projective hierarchies at which it consistently fails. Recall that a subset $A$ of a Polish space $X$ has the $\mathsf{PSP}$ if it is either countable or ${}^{\omega}2$ embeds into $A$. Its generalised variant asserts that a subset $A$ of a $\lambda$-Polish space $X$ has the $\lambda\hyp\mathsf{PSP}$ if either it is of size less than or equal to $\lambda$ or ${}^{\lambda}2$ embeds into $A$ as a closed-in-$X$ set. It is a well known result of Souslin that all analytic sets have the $\mathsf{PSP}$. In the generalised setting at singular cardinals of countable cofinality, Dimonte and Motto Ros have proved that all $\lambda$-analytic sets have the $\lambda\hyp\mathsf{PSP}$ as well. 
    
    Extending the $\mathsf{PSP}$ and the $\lambda\hyp\mathsf{PSP}$ further in the projective hierarchy involves, in both cases, large cardinals. In the classical case, the consistency strength is quite low: ``all the coanalytic sets have the $\mathsf{PSP}$" is equiconsistent to the existence of an inaccessible cardinal, as it is ``all the projective sets have the $\mathsf{PSP}$". Even below that, G\"{o}del proved that in $L$ there is a (lightface) $\Pi^1_1$ set without the $\mathsf{PSP}$, therefore ``all the $\Pi^1_1$ sets have the $\mathsf{PSP}$" cannot be a theorem of $\mathsf{ZFC}$, assumed its consistency.
    
    The generalised case is still under study: Dimonte and Motto Ros have proved that in $L$ there is a $\lambda$-coanalytic subset of ${}^{\lambda}2$ without the $\lambda\hyp\mathsf{PSP}$, and that the lower bound of all $\lambda$-coanalytic sets of $\lambda$-reals having the $\lambda\hyp\mathsf{PSP}$ is $0^\sharp$ (\cite[Corollary 7.2.11]{DMR}), a bound that, as the authors noted, can be easily raised up to $a^\sharp$ for every $a\in{}^{\lambda}2$; L\"{u}cke and M\"{u}ller proved that ``all the sets that are definable with a $\lambda\hyp\Sigma^1_2$ formula with parameters in $H_\lambda\cup\{\lambda\}$ have the $\lambda\hyp\mathsf{PSP}$" is equiconsistent to the existence of an inner model with $\omega$-many measurable cardinals (\cite{LM}); as an all-encompassing upper bound, recent work by Dimonte, Poveda and Thei shows that, assuming a $<\theta$-supercompact cardinal $\lambda$ with $\lambda<\theta$ and $\theta$ inaccessible, there exists a model of $\mathsf{ZFC}$ where $\lambda$ is a strong limit cardinal of countable cofinality and every set in $\mathcal{P}({}^{\omega}\lambda)\cap L(V_{\lambda+1})$ has the $\lambda\hyp\mathsf{PSP}$ (\cite[Main Theorem 1]{dimonte2024baire}).

    In this paper, we improve the stated lower bound in two distinct ways. First, we establish the following result, which is slightly stronger than a formulation that only concerns $\lambda$-coanalytic sets: 
    \begin{namedthm}{Corollary \ref{c310}}
	If there is an uncountable strong limit cardinal $\lambda$ of countable cofinality such that all \emph{(lightface)} $\lambda\hyp\Pi^1_1$ subsets of ${}^{\lambda}2$ have the $\lambda\hyp\mathsf{PSP}$, then there is an inner model with a measurable cardinal.
    \end{namedthm}

    Moving further up the large cardinal hierarchy requires us to shift from the class of (lightface) $\lambda\hyp\Pi^1_1$ subsets of ${}^{\lambda}2$ to the broader class of $\lambda$-coanalytic sets. We prove the following:
    
    \begin{namedthm}{Corollary \ref{theorem:main-theorem}}
	If there is an uncountable strong limit cardinal $\lambda$ of countable cofinality such that all $\lambda$-coanalytic subsets of ${}^{\lambda}2$ have the $\lambda\hyp\mathsf{PSP}$, then $0^{\dagger}$ exists.
    \end{namedthm}

    The paper is organised as follows. In Section \ref{section:preliminaries}, we recall the basics of $\mathsf{GDST}$ at singular cardinals of countable cofinality and review all necessary background results. Here, we collect the description of the $\lambda$-projective and $\lambda$-analytical hierarchies together with some results on codes for structures of size $\lambda$. This section also includes a toy example which, although not framed in a descriptive set-theoretic context, illustrates how core models can be used to obtain consistency strength lower bounds. Sections \ref{section:dodd-jensen} and \ref{section:0-dagger} deal with inner model theoretic notions at the level of the existence of an inner model with a measurable cardinal and $0^\dagger$. As these topics may be unfamiliar to part of our audience, we try to provide as much detail as possible. In Section \ref{section:dodd-jensen}, we study codes for Dodd-Jensen premice and mice of size $\lambda$ and obtain their (generalised) descriptive complexity, as well as the descriptive complexity of other sets relevant to our study. 
    Here, Corollary \ref{c310} is proved. In Section \ref{section:0-dagger}, we extend the previous discussion to $L[U]$, and prove Corollary \ref{theorem:main-theorem}. Section \ref{section:final-remarks} concludes with some final remarks.\newline

    \textbf{Acknowledgments.} The first and second authors were partially supported by \textsf{Progetto PRIN 2022 – Models, Sets and Classifications – Code no. 2022TECZJA CUP G53D23001890006, funded by the European Union – NextGenerationEU – PNRR M4 C2 I1.1}. The first author has also been partially supported by the \textsf{OeAD Ernst Mach grant - worldwide}. The first author would like to thank Omer Ben-Neria for suggesting, at the beginning of this project, the use of the Dodd-Jensen core model. He also wishes to thank Philipp Lücke for valuable discussions during the \emph{Determinacy, Inner Models and Forcing Axioms} workshop held at the Erwin Schrödinger International Institute for Mathematics and Physics (ESI), Vienna, in June 2024, and several more after. The first and third authors gratefully acknowledge the support of the ESI during this event. This research was partially supported by GNSAGA INdAM, via the funding of the \emph{XXVIII Incontro di Logica AILA}, held at the University of Udine in September 2024. This research was funded in whole, or in part, by the \textsf{Austrian Science Fund (FWF) [10.55776/Y1498, 10.55776/I6087, 10.55776/STA139]}. For the purpose of open access, the authors have applied a CC BY public copyright licence to any Author Accepted Manuscript version arising from this submission.

    \section{Preliminaries}\label{section:preliminaries}

    The reader of this paper may be unfamiliar to most of the literature in $\mathsf{GDST}$ at singular cardinals of countable cofinality. Within the first two subsections we try to remediate this. The third subsection provides the reader with a toy example on how core model theory is used for obtaining consistency strength lower bounds.

    \subsection{\textsf{GDST} at singular cardinals of countable cofinality} Let $\lambda$ be an infinite cardinal and let $X$ be a $\lambda$-Polish space, i.e., a completely metrizable topological space with weight\footnote{The weight of a topological space $X$ is $\kappa$ if every dense subset $Y$ of $X$ is of size greater than or equal to $\kappa$.} $wt(X)\leq\lambda$:

    \begin{definition}[{\cite[Definition 7.0.1]{DMR}}]
	       A set $A\subseteq X$ has the \emph{$\lambda$-Perfect
           Set Property} (or $\lambda$-$\mathsf{PSP}$) if either $|A|\leq\lambda$ or $^{\lambda}2$ embeds into $A$ as a closed-in-$X$ set.
    \end{definition}

    In this paper, we restrict attention to the case $X = {}^{\lambda}2$. If $\lambda$ is a strong limit cardinal of countable cofinality, then ${}^{\lambda}2$, endowed with the bounded topology, is a $\lambda$-Polish space (see \cite[Proposition~3.12 (g,h)]{AndrettaMottoRos}). This space is called the \emph{generalised Cantor space}.

    \begin{proviso}
	Throughout this chapter and unless otherwise specified, $\lambda$ is assumed to be a strong limit cardinal of countable cofinality.
    \end{proviso}

    Given a topological space $(X,\tau)$, a set $B\subseteq X$ is \emph{$\lambda^{+}$-Borel} if it belongs to the $\lambda^{+}$-algebra $\lambda$\textbf{-Bor} generated by the open sets of $X$. This $\lambda^+$-algebra can be stratified as in the classical case: let $\lambda\hyp\boldsymbol{\Sigma}^0_1$ denote the class of open sets in $X$, and let $\lambda\hyp\boldsymbol{\Pi}^1_0$ be the class of closed sets in $X$. Then, for each $\alpha<\lambda^+$, let $\lambda\hyp\boldsymbol{\Sigma}^0_\alpha=\{\bigcup_{\beta<\lambda}A_\beta:X\setminus  A_\beta\in\lambda\hyp\boldsymbol{\Sigma}^1_\delta\wedge \delta<\alpha\}$, and $\lambda\hyp\boldsymbol{\Pi}^0_\alpha=\{X\setminus A:A\in\lambda\hyp\boldsymbol{\Sigma}^0_\alpha\}$. Finally, let $\lambda\hyp\boldsymbol{\Delta}^0_\alpha=\lambda\hyp\boldsymbol{\Sigma}^0_\alpha\cap\lambda\hyp\boldsymbol{\Pi}^0_\alpha$. Then, the class of $\lambda$-Borel sets is:
    \begin{equation*}
	\bigcup_{\alpha<\lambda^+}\lambda\hyp\boldsymbol{\Sigma^0_\alpha}=\bigcup_{\alpha<\lambda^+}\lambda\hyp\boldsymbol{\Pi^0_n}=\bigcup_{\alpha<\lambda^+}\lambda\hyp\boldsymbol{\Delta^0_n}.
\end{equation*}
    When $\lambda$ is singular, we can drop the $+$ from the above. If $X$ is a $\lambda$-Polish space, a set $A\subseteq X$ is \emph{$\lambda$-analytic} if it is the continuous image of some $\lambda$-Polish space $Y$, and it is \emph{$\lambda$-coanalytic} if it is the complement of a $\lambda$-analytic set. The class of all $\lambda$-analytic (resp. $\lambda$-coanalytic) subsets of $X$ is denoted by $\lambda\hyp\boldsymbol{\Sigma^1_1}$ (resp. $\lambda\hyp\boldsymbol{\Pi^1_1}$). Subsets of $X$ that are both $\lambda$-analytic and $\lambda$-coanalytic are \emph{$\lambda$-bianalytic}. The class of all $\lambda$-bianalytic subsets of $X$ is denoted by $\lambda\hyp\boldsymbol{\Delta^1_1}$ and, as in the classical case, it is just $\lambda$\textbf{-Bor} (see {\cite[Theorem 5.3.5]{DMR}}). These classes can be recursively extended: for each $n\geq 1$,
    $\lambda\hyp\boldsymbol{\Sigma^1_{n+1}}$ denotes the set of those subsets $A$ of $X$ for which there is some $\lambda$-Polish space $Y$ such that there exists a continuous function $f:Y\longrightarrow X$ and a $\lambda\hyp\boldsymbol{\Sigma^1_n}$-subset $B$ of $Y$ such that $f(B)=A$. Then, for each $n\geq 1$, one defines $\lambda\hyp\boldsymbol{\Pi^1_n}=\{A\subseteq X:X-A\in\lambda\hyp\boldsymbol{\Sigma^1_n}\}$ and $\lambda\hyp\boldsymbol{\Delta^1_n}=\lambda\hyp\boldsymbol{\Sigma^1_n}(X)\cap\lambda\hyp\boldsymbol{\Pi^1_n}$. If there is some $n<\omega$ such that $A\subseteq X$ belongs to some $\lambda\hyp\boldsymbol{\Sigma^1_n}$, then $A$ is \emph{$\lambda$-projective}. Thus, the class of $\lambda$-projective sets is the following:
    \begin{equation*}
	\bigcup_{n\geq 1}\lambda\hyp\boldsymbol{\Sigma^1_n}=\bigcup_{n\geq1}\lambda\hyp\boldsymbol{\Pi^1_n}=\bigcup_{n\geq1}\lambda\hyp\boldsymbol{\Delta^1_n}.
	\end{equation*}
	Both the $\lambda$-Borel and $\lambda$-projective hierarchies can be made finer. In what follows, we work in the language of set theory allowing both first- and second-order variables, the former ranging over $\lambda$, the latter over ${}^{\lambda}2$. If $Q$ is either $\exists$ or $\forall$, $Q^0$ (sometimes $Q$ alone, if there is no danger of confusion) indicates that the quantifier ranges over first order variables; by $Q^{1}$ we mean that the quantifier ranges over second order variables. Let $x\in{}^{\lambda}2\cup H_\lambda$, and let $i\in\{0,1\}$. A set $A\subseteq({}^{\lambda}2)^{k}$, with $k\in\omega$, is (lightface) $\lambda\hyp\Sigma^i_n(x)$ if and only if there is a formula $\varphi$ such that for all $y\in{}^{\lambda}2$, $y\in A$ if and only if 
		\[
			\exists^i z_1\forall^i z_2\ldots Q^i z_n\langle H_\lambda;\in,x,y,z_1,\ldots,z_n\rangle \vDash\varphi(x,y,z_1,\ldots,z_n).
		\]
	A set $A\subseteq ({}^{\lambda}2)^{k}$ is (lightface) $\lambda\hyp\Pi^i_n(x)$ if and only if it is the complement of a $\lambda\hyp\Sigma^i_n(x)$ set, and it is $\lambda\hyp\Delta^i_n(x)$ if it is both $\lambda\hyp\Sigma^i_n(x)$ and $\lambda\hyp\Pi^i_n(x)$. We write $\lambda\hyp\Delta^i_n$ for $\lambda\hyp\Sigma^i_n\cap\lambda\hyp\Pi^i_n$ for every $n\in\omega$ and $i\in\{0,1\}$. 
	Observe that for every $n$ and $i\in\{0,1\}$, the collections $\lambda\hyp\Sigma^i_n(x)$ and $\lambda\hyp\Pi^i_n(x)$ are dual, and that $\lambda\hyp\Sigma^i_n(x)\cup\lambda\hyp\Pi^i_n(x)\subseteq\lambda\hyp\Sigma^i_{n+1}(x)\cap\lambda\hyp\Pi^i_{n+1}(x)$. For the second claim, simply note that, for every $n\in\omega$ and $i\in\{0,1\}$, $\lambda\hyp\Sigma^i_n(x)\subseteq\lambda\hyp\Pi^i_{n+1}(x)$, and that $\lambda\hyp\Pi^i_n(x)\subseteq\lambda\hyp\Sigma^i_{n+1}(x)$. A set that is (lightface) $\lambda\hyp\Sigma^i_{n+1}(x)$ is the projection of a (lightface) $\lambda\hyp\Pi^i_n(x)$. A set that is $\lambda\hyp\Sigma^i_n(x)$ for some $n\in\omega$ and some $x\in{}^{\lambda}2$ is \emph{$\lambda$-arithmetical} if $i=0$ and \emph{$\lambda$-analytical} if $i=1$. The $\lambda$-arithmetical and $\lambda$-analytical hierarchies relate to the $\lambda$-Borel and $\lambda$-projective hierarchies as follows: 
	
	\begin{proposition}\label{p24} Let $\lambda$ be a singular cardinal of countable cofinality and suppose that $A\subseteq {}^{\lambda}2$. Then:
		\begin{itemize}
			\item[\emph{(1)}] $A\in\lambda\hyp\boldsymbol{\Sigma^0_n}$ if and only if $A\in\lambda\hyp\Sigma^0_n(x)$ for some $x\in {}^{\lambda}2$, and similarly for $\lambda\hyp\boldsymbol{\Pi^0_\alpha}$.
			\item[\emph{(2)}] $A\in\lambda\hyp\boldsymbol{\Sigma^1_n}$ if and only if $A\in\lambda\hyp\Sigma^1_n(x)$ for some $x\in {}^{\lambda}2$, and similarly for $\lambda\hyp\boldsymbol{\Pi^1_n}$.
		\end{itemize}
	\end{proposition}

    Note that in defining the relative lightface $\lambda$-projective hierarchy we permitted both ``first-order'' and ``second-order'' parameters, respectively in $H_\lambda$ and in ${}^\lambda 2$. This is new respect to the classical lightface projective hierarchy, as in that case the first-order parameters were simply definable without parameters and therefore redundant. The results in \cite{dimonte2023descriptive} show that the second-order parameters have much more strength then first-order ones, but it is still open whether first-order parameters have a meaningful impact in the construction of the hierarchy. Corollary \ref{c310} and Corollary \ref{theorem:main-theorem} seem to indicate so.

 \subsection{Coding structures} Let $\eta$ be a cardinal. Then every $z\in{}^{\eta}2$ encodes a binary relation $E_{z}$ on $\eta$ given by $\langle\alpha,\beta\rangle\in E_{z}$ if and only if $z(\langle\alpha,\beta\rangle)=1$, where $\langle\cdot , \cdot\rangle$ is an ordinal pairing function. We choose the pairing function so that it is definable in $H_\eta$ without parameters and cardinals are closed under it.

 For any $z\in{}^{\eta}2$, if $E_z$ on $\eta$ is a well-order, then its transitive collapse will be an ordinal less than $\eta$. Viceversa, if $\alpha$ is an ordinal less then $\eta$, then there exists a $z\in{}^{\eta}2$ such that the ordertype of $(z,E_z)$ is $\alpha$.

 Recall that every well-founded and extensional structure is isomorphic to a unique transitive structure, its Mostowski or transitive collapse. For any $z\in{}^{\eta}2$, if $E_z$ on $\eta$ is well-founded and extensional we denote by $tr(\eta,\in_{z})$ the transitive collapse\footnote{As noted in \cite[Remark 7.2.3]{DMR}, the word \emph{collapse} can be misleading here, for the collapse of an ordinal $\alpha < \eta$ may be an arbitrary ordinal in $[\eta, \eta^+)$. For further details, we refer the reader to that source.} of the structure $(\eta,\in_{z})$, and with $\pi_{z}$ we refer to the inverse of its collapse function.
Conversely, let $(M,\in)$ be a transitive model of size $\eta$, let $\pi:M\to\eta$ be a bijection, and let $E_{\pi}$ be the relation on $\eta$ given by $E_{\pi}(\alpha,\beta)$ if and only if $\pi^{-1}(\alpha)\in \pi^{-1}(\beta)$. It is clear that $\pi:(M,\in)\to(\eta,E_{\pi})$ is an isomorphism. In fact, $(M,\in)$ is the transitive collapse of $(\eta,E_{\pi})$. 
 
Define:
        \[
            \mathsf{WO}_{\lambda}:=\{z\in{}^{\lambda}2:z\text{ codes a well-order on $\lambda$}\}.
        \]
    It is routine to check that ``$x\in{}^{\lambda}2$ codes a linear order" is $\Pi^0_1$. In turn, to say that the coded order is well-founded can be expressed by the formula $\forall y\in{}^{\lambda}2\ \exists\alpha\in\lambda\ \forall\beta\in\lambda(y(\alpha)=1\rightarrow(x(\langle\beta,\alpha\rangle)=1\rightarrow y(\beta)=0)$
Call $\varphi(x)$ the conjunction of the previous two formulas. Then, 
\[
\mathsf{WO}_\lambda=\{x\in {}^\lambda 2:{}^\lambda 2\vDash \varphi(x)\},
\]
where $\varphi(x)$ is a $\Pi^1_1$-formula. It thus follows that:
\begin{proposition}\label{WO_is_lightface_pi_1_1} $\mathsf{WO}_{\lambda}$ is \emph{(lightface)} $\lambda\hyp\Pi^1_1$. 
\end{proposition}

From now on, if $z\in \mathsf{WO}_{\lambda}$ we write $<_{z}$ instead of $E_{z}$. 

Every set of ordinals coded by elements in a $\lambda$-analytic subset of $\mathsf{WO}_{\lambda}$ is bounded in $\lambda^{+}$:

\begin{lemma}[Boundedness Lemma {\cite[Theorem 6.2.3]{DMR}}]\label{lemma:boundedness-lemma} Let $\lambda$ be such that $2^{<\lambda}=\lambda$. If $A\subseteq \mathsf{WO}_{\lambda}$ is $\lambda\hyp\boldsymbol{\Sigma^1_1}$, then $\sup\{ot(<_{x}):x\in A\}<\lambda^{+}$.
\end{lemma}

Let $\mathsf{EW}_{\lambda}$ denote the set of codes for extensional and well-founded structures of size $\lambda$. Observe that extensionality can be expressed by a $\Pi^0_1$-formula, for  $x\in{}^{\lambda}2$ codes an extensional relation if 
$\forall\alpha\in\lambda\ \forall\beta\in\lambda\ \forall\delta\in\lambda(\alpha=\beta\leftrightarrow[x(\langle\delta,\alpha\rangle)=1\leftrightarrow x(\langle\delta,\beta\rangle)=1])$. Since, as shown above, well-foundedness can be expressed by a $\Pi^1_1$-formula, it follows that:

\begin{proposition}\label{EW_is_lightface_pi_1_1}
	$\mathsf{EW}_{\lambda}$ is \emph{(lightface)} $\lambda\hyp\Pi^1_1$.
\end{proposition}

If $z \in \mathsf{EW}_\lambda$, we use the notation $\in_z$ in place of $E_z$.

We want now to describe ``$\pi_x(b)=\alpha$'', with $x \in \mathsf{EW}_\lambda$, $b\in H_\lambda$ and $\alpha<\lambda$, in the easiest possible way (i.e., with a formula of the lowest complexity). This holds if there is a function $f:\lambda\to H_\lambda$ that is an isomorphism between $(trcl_x(\alpha),\in_x)$ and $(trcl(b),\in)$. We need therefore to code functions from $\lambda$ to $H_\lambda$.

If $b\in H_\lambda$ then $|trcl(b)|<\lambda$, so there exists a $\eta<\lambda$ cardinal and a $z\in{}^\eta 2$ such that $(\eta,\in_z)$ is isomorphic to $(trcl(b),\in)$. We can therefore code $trcl(b)$ with an $x\in{}^\lambda 2$, simply by defining $x\upharpoonright\eta=z$, and for any $\alpha\geq\eta$ $x(\alpha)=0$. The definition of such a code is: $\exists \eta<\lambda$ such that $\eta$ is a cardinal, $(\eta,\in_{x\upharpoonright\eta})$ is well-founded and extensional, and $\forall \alpha\geq\eta\, x(\alpha)=0$. Well-founded and extensionality are defined as for $\mathsf{EW}_\lambda$: extensionality by $\forall\alpha\in\eta\ \forall\beta\in\eta\ \forall\delta\in\eta(\alpha=\beta\leftrightarrow[x(\langle\delta,\alpha\rangle)=1\leftrightarrow x(\langle\delta,\beta\rangle)=1])$, that is a $\lambda\hyp\Delta^0_0(\eta)$ formula, and well-foundedness by $\forall y\in{}^{\eta}2\ \exists\alpha\in\eta\ \forall\beta\in\eta(y(\alpha)=1\rightarrow(x(\langle\beta,\alpha\rangle)=1\rightarrow y(\beta)=0)$, that is again $\lambda\hyp\Delta^0_0(\eta)$. Moreover, ``$\eta$ is a cardinal'' is definable in $H_\lambda$. 

Finally, we define $F(\lambda,H_\lambda)$ as the set of $x\in{}^\lambda 2$ that code $\lambda$-sequences of elements as above, i.e.:
\begin{gather*}
    x\in F(\lambda, H_\lambda)\quad\text{iff}\quad \forall\alpha<\lambda\ \exists \eta<\lambda (\eta\text{ is a cardinal }\land\\
    \land \forall\beta\in\eta\ \forall\gamma\in\eta\ \forall\delta\in\eta(\beta=\gamma\leftrightarrow[x(\langle\alpha,\delta,\beta\rangle)=1\leftrightarrow x(\langle\alpha,\delta,\gamma\rangle)=1])\land\\
    \land \forall y\in{}^{\eta}2\ \exists\beta\in\eta\ \forall\gamma\in\eta(y(\beta)=1\rightarrow(x(\langle\alpha,\gamma,\beta\rangle)=1\rightarrow y(\gamma)=0)\land\\
    \land\forall\beta\geq\eta\ x(\langle\alpha,\beta\rangle)=0.
\end{gather*}

The set $F(\lambda, H_\lambda)$ is therefore $\lambda\hyp\Sigma^1_0$.

When $x\in F(\lambda,H_\lambda)$, call $(x)_\alpha$ the element of ${}^\lambda 2$ that codes the $\alpha$-the element of the sequence, i.e., $(x)_\alpha(\beta)=x(\langle\alpha,\beta\rangle)$ for any $\alpha,\beta<\lambda$. Then the $\eta$ that witnesses this is unique, otherwise $(x)_\alpha\upharpoonright\eta$ would not be extensional, and therefore also the transitive $c\in H_\lambda$ that is coded by it is unique. This can be all defined in $H_\lambda$: the transitive set $c\in H_\lambda$ is coded by $(x)_\alpha$ iff there exists an isomorphism $f$ between $(\eta,\in_{(x)_\alpha\upharpoonright \eta})$ and $(c,\in)$, i.e., $f:\eta\to c$ is a bijection and $\forall\beta,\gamma<\eta (x(\langle\alpha,\beta,\gamma\rangle)=1 \leftrightarrow f(\beta)\in f(\gamma))$. We define $c(x,\alpha)$ as the set coded by $(x)_\alpha$.

For any $x\in\mathsf{EW}_\lambda$ and any $\alpha<\lambda$, the set $trcl_x(\alpha)$ is easily seen as $\lambda\hyp\Sigma^1_0$-definable, as $\beta\in trcl_x(\alpha)$ iff there exists a finite sequence $\langle\alpha_0,\dots,\alpha_n\rangle$ such that $\alpha_0=\beta$, $\alpha_n=\alpha$, and for every $0\leq i<n$ $x(\langle\alpha_i,\alpha_{i+1}\rangle)=1$. On the other hand, if $b\in H_\lambda$ then also $trcl(b)$ is easily definable in $H_\lambda$ in a similar way. 

To motivate the following definition, observe that in the classical setting, where $\lambda = \omega$, verifying whether $\pi^{-1}_z(m) = n$ for some $m,n \in \omega$ amounts to finding an $\in$-preserving bijection between the $\in_z$-predecessors of $n$ and $m$. Since the $\in_z$-predecessors of $n$ are bounded in $\omega$, such a bijection must go from some $n'<\omega$ to $m$. There are countable many of such bijections, so the set $\{z\in{}^{\omega}2:\pi_{z}^{-1}(m)=n\}$ is $\omega\hyp$Borel, or (lightface) $\omega\hyp\Sigma^1_0$ with $m$ and $n$ as parameters. This is not the case anymore in the generalised context: it might happen that the set of $\in_{z}$-predecessors of some $\alpha<\lambda$ isn't bounded in $\lambda$, so we need a full function from $\lambda$ to $\lambda$ to check whether $\pi^{-1}_{z}(\alpha)<\lambda$, and this goes beyond $\lambda$-Borelness: it is indeed $\lambda$-analytic relative to $\mathsf{EW}_\lambda$ (see \cite[Lemma 7.2.4]{DMR}). Naturally, we are interested in maintaining as low as possible the descriptive complexity of the sets involved in our computations. In order to do so, we impose a boundedness condition on the set of codes for extensional and well-founded structures:

\begin{definition}[{\cite[Definition 7.2.5]{DMR}}]
	Let $z\in{}^{\lambda}2$ be a code for an extensional and well-founded structure. We say that $(\lambda,\in_{z})$ has a \emph{bounded} collapse if and only if for each $b\in H_{\lambda}$, $\pi^{-1}_{z}(b)$, when it exists, has boundedly many $\in_{z}$-predecessors.
\end{definition}

We denote by $\mathsf{BC}_{\lambda}$ the set of codes for such bounded structures. As the following shows, it is safe to consider only codes for bounded, extensional and well-founded structures:

\begin{proposition}[{\cite[Lemma 7.2.7]{DMR}}] For any transitive structure $(M,\in)$ of size $\lambda$ there exists a $z\in\mathsf{BC}_{\lambda}$ such that $(M,\in)\cong(\lambda,\in_{z})$. 
\end{proposition}

Moreover:

\begin{proposition}\label{prop:BC-is-pi-11}
	$\mathsf{BC}_\lambda$ is \emph{(lightface)} $\lambda\hyp\Pi^1_1$. 
\end{proposition}
\begin{proof}
	First note that $x\in\mathsf{BC}_\lambda$ if and only if $x\in\mathsf{EW}_\lambda\wedge \forall\alpha\in\lambda\forall b\in H_\lambda[\pi^{-1}_x(b)=\alpha\rightarrow\exists\gamma\in\lambda\forall\delta\in\lambda(\delta\in trcl_x(\alpha)\rightarrow\delta\in\gamma)]$. By Proposition \ref{EW_is_lightface_pi_1_1}, $x\in\mathsf{EW}_\lambda$ is (lightface) $\lambda\hyp\Pi^1_1$. For the rest, observe that $\pi^{-1}_x(\alpha)=b$ is expressible by a $\Sigma^1_1$-formula with $x,\alpha$ and $b$ as parameters, for it holds if and only if there is a function $f\in{}^{\lambda}H_\lambda$ such that for every $\beta,\gamma<\lambda$ one has that $f\upharpoonright trcl_x(\alpha):trcl_x(\alpha)\rightarrow trcl(b)$ is a bijection (where, again, $trcl_x(\alpha)$ is defined as $trcl(b)$ above with $\in_x$ instead of $\in$); that if $\beta,\gamma\in trcl_x(\alpha)$, then $\beta\in_x\gamma\leftrightarrow f(\beta)\in f(\gamma)$; and that if $\beta\notin trcl_x(\alpha)$, then $f(\beta)=0$. More precisely:
    \begin{gather*}
      \exists z\in{}^\lambda 2\ (z\in F(\lambda,H_\lambda) \land \forall a\in trcl(b)\ \exists \beta<\lambda (\beta\in trcl_x(\alpha) \land c(z,\beta)=trcl(a))\land\\
      \land \forall\beta\in trcl_x(\alpha)\ \forall\gamma\in trcl_x(\alpha) (c(z,\beta)=c(z,\gamma)\rightarrow\beta=\gamma)\land\\
      \land  \forall\beta\in trcl_x(\alpha)\ \forall\gamma\in trcl_x(\alpha) ( x(\langle \beta,\gamma \rangle)=1 \leftrightarrow c(z,\beta)\in c(z,\gamma))\land\\
      \land \forall \beta\notin trcl_x(\alpha)\ c(z,\beta)=0).
    \end{gather*}
\end{proof}

We close this subsection with some useful results:

\begin{lemma}\label{lemma:complexity_of_satisfaction}
    Let $\Phi$ be a finite family of $\mathcal{L}_\in(\dot{c}_1,\ldots,\dot{c}_n)$-sentences, with $\dot{c}_1,\ldots,\dot{c}_n$ a family of extra symbols. Then, $\{(z_0,\ldots,z_n)\in{}^{n+1}(^{\lambda}2):(\lambda,\in_{z_0},z_1,\ldots,z_n)\vDash\Phi\}$ is \emph{(lightface)} $\lambda\hyp\Delta^1_1$ in $^{n+1}({}^{\lambda}2)$.
\end{lemma}
\begin{proof}
    For a reference, see \cite[Proposition 13.8]{Kanamori}. The proof is quite straightforward, although lengthy. A proof for a family $\Phi$ of $\mathcal L_\in$-sentences can be found in \cite[Chapter 3. \textsection 5]{drake1974introduction}. The proof easily generalises to our case. 
\end{proof}

\begin{lemma}\label{lemma:y-in-transitive-closure-of-lambda-inx-is-pi11}
    The set $\{(x,y)\in \mathsf{BC}_{\lambda}\times{}^{\lambda}2:y\in tr(\lambda,\in_x)\}$ is \emph{(lightface)} $\lambda\hyp\Pi^1_1$ in ${}^{\lambda}2\times{}^{\lambda}2$.
\end{lemma}
\begin{proof}
    Simply observe that $y\in tr(\lambda,\in_x)$ if and only if $\exists\alpha<\lambda(\pi^{-1}_{x}(\alpha)=y)$. 
\end{proof}

\subsection{A toy example}

Roughly speaking, a core model $K$ is a definable structure that is, under suitable anti-large cardinal assumptions, very close to $V$ in the sense that every uncountable set in $V$ is covered by a set in $K$ of the same size. If $K$ is a model of \textsf{AC}, this implies that $(\lambda^{+})^{K}=\lambda^{+}$ for every singular cardinal $\lambda$. An inner model satisfying the first property is said to have the full covering property, the second property is known as the weak covering property (see Definition \ref{def:covering-properties}). 

Core models are useful to fix consistency strength lower bounds, and the strategy for that is quite standard. If one is to fix the consistency strength lower bound of an arbitrary statement $\varphi$, one first shows that $\neg\varphi$ holds in the core model. As said, under the corresponding anti-large cardinal assumptions the core model is close to $V$, and this closeness might serve to show that $\neg\varphi$ also holds in $V$. It then follows that if $\varphi$ is true in $V$, these anti-large cardinal assumptions must fail, for otherwise one would get a contradiction. 

Typically, core models satisfy Global Choice, i.e., they have a definable well-order. We can use such well-order and the parameter $\lambda$ to define subsets of ${}^\lambda 2$. In the following, we will use the term $\Sigma_n(\lambda)$-definable (resp. $\Pi_n(\lambda)$ and $\Delta_n(\lambda)$) to indicate a set that is $\Sigma_n$-definable (resp. $\Pi_n$ and $\Delta_n$) with parameter $\lambda$.

For example, Corollary \ref{corollary:toy-example-Delta_1-definable-subset-wo-psp} below, consequence of Proposition \ref{proposition:toy-example}, says that, in $L$, there is a $\Pi_{1}(\lambda)$-definable (in fact, $\Delta_1(\lambda)$-definable) subset $A$ of ${}^{\lambda}2$ that does not have the $\lambda$-$\mathsf{PSP}$. By the Covering Theorem for $L$, if $0^{\sharp}$ doesn't exist, then every uncountable set in $V$ is covered by a set in $L$ of the same size. This in turn implies, being $L$ a model of choice, that if $0^{\sharp}$ doesn't exist, then $L$ computes correctly the successor of every singular cardinal, which is enough to prove that, under that same assumption, the set $A$, as seen from $V$, is also a $\Pi_2(\lambda)$-definable subset of ${}^{\lambda}2$ without the $\lambda\hyp\mathsf{PSP}$. It thus follows that if every $\Pi_{2}(\lambda)$-definable subset of $^{\lambda}2$ in $V$ has the $\lambda$-$\mathsf{PSP}$, then $0^{\sharp}$ must exist, as written in Corollary \ref{corollary:toy-example-if-0-sharp-pi2-set-without-psp}. 

\begin{proposition}\label{proposition:toy-example}
	Let $K$ be an inner model with a $\Pi_{n}$-definable well order $\leq_{K}$ with $n\geq1$. Then, there is a set $A\subseteq {}^{\lambda}2$ in $K$ that is $\Pi_n(\lambda)$-definable and without the $\lambda$-$\mathsf{PSP}$.
\end{proposition}
\begin{proof} 
	We work in $K$. Let $A_{K}=\{x\in \mathsf{WO}_{\lambda}:\forall y\,((\lambda,\in_x)\cong(\lambda,\in_y)\rightarrow x\leq_{K}y)\}$. The statement $``(\lambda,\in_x)\cong(\lambda,\in_y)"$ can be expressed as a $\Sigma_1$-formula with parameter $\lambda$, say $\exists w\varphi(x,y,w,\lambda)$, and $``x\leq_{K}y"$ is by assumption  $\Pi_{n}$-definable with $n\geq 1$, so say that $``x\leq_{K}y"$ is of the form $$\forall x_{n-1}\exists x_{n-2}\forall x_{n-3}\ldots\psi(x,y,x_0,\ldots,x_{n-1}).$$ The formula $x\in \mathsf{WO}_{\lambda}$ is $\Pi_1(\lambda)$, so it can be written as $\forall z_1\theta(x,z_1,\lambda)$ with $\theta(x,z_1,\lambda)$ a $\Delta_0$-formula. It then follows that
	\begin{equation*}
		\begin{split}
			x\in A_{K} \leftrightarrow& \,\forall z_1\forall y\forall w\,\forall x_{n-1}\exists x_{n-2}\forall x_{n-3}\ldots \\&(\theta(x,z_1,\lambda)\wedge(\neg\varphi(x,y,w,\lambda)\vee\psi(x,y,x_{0},\ldots,x_{n-1}))),
		\end{split}
	\end{equation*}
	therefore $A_{K}$ is $\Pi_n(\lambda)$-definable.
	
	Now, for every $\alpha<\lambda^{+}$ there is some $x\in {}^{\lambda}2$ coding a well-order on $\lambda$ of order type $\alpha$. By definition, there is a unique code in $A_{K}$ for each $\alpha<\lambda^{+}$, so $|A_{K}|=\lambda^{+}$. Therefore, if $A_{K}$ had the $\lambda$-$\mathsf{PSP}$, there would exist an embedding $\pi:{}^{\lambda}2\to A_{K}$. Since $\pi[{}^{\lambda}2]$ is $\lambda$-analytic and $|\pi[{}^{\lambda}2] |\geq \lambda^{+}$, the set $\{ot(<_{x}):x\in\pi[{}^{\lambda}2]\}$ would be unbounded in $\lambda^{+}$. But this contradicts the Boundedness Lemma, so $A_{K}$ does not have the $\lambda$-$\mathsf{PSP}$. 
\end{proof}

\begin{corollary}
  Let $K$ be an inner model with a $\Sigma_{n}$- or $\Delta_n$-definable well order $\leq_{K}$ with $n\geq1$. Then, there is a set $A\subseteq {}^{\lambda}2$ in $K$ that is $\Pi_n(\lambda)$-definable and without the $\lambda$-$\mathsf{PSP}$.
\end{corollary}
\begin{proof}
    Note that any $\Pi_n$-definable linear order is also $\Sigma_n$-definable, and viceversa, therefore Proposition \ref{proposition:toy-example} works for $K$.
\end{proof}

Thus, since $L$ admits a $\Sigma_1$-definable well-ordering, Proposition \ref{proposition:toy-example} already implies that there is a $\Pi_1(\lambda)$-definable subset of ${}^{\lambda}2$ in $L$ without the $\lambda$-$\mathsf{PSP}$. In this case, however, we can do better. Recall that a well-order $\leq$ on a transitive model $M$ of \textsf{ZFC} is said to be $\Delta_1$-good if for every $x\in M$ the relation $z(x):=\{y:y\leq x\}$ is $\Delta_1$. It is well known that $L$ admits a $ \Delta_1$-good well-ordering $\leq_L$.  
Moreover, note that $x\in A_L$ if and only if
\[
x\in\mathsf{WO}_\lambda\wedge\exists z(z=\{y\in{}^{\lambda}2:y\leq_{L}x\}\wedge \forall y\in z((\lambda,\in_x)\ncong(\lambda,\in_y))),
\]
which is easily seen to be $\Delta_1(\lambda)$ (well-foundedness is $\Delta_1$, for it is $\Pi_1$ via the non-existence of an infinite descending chain and $\Sigma_1$ via the existence of a ranking function). That is:

\begin{corollary}\label{corollary:toy-example-Delta_1-definable-subset-wo-psp} In $L$, there is a $\Delta_1(\lambda)$-definable subset of ${}^{\lambda}2$ without the $\lambda\hyp\mathsf{PSP}$.
\end{corollary}

Now, let $A_{L}^{V}$ be the set
$$\{x\in \mathsf{WO}_{\lambda}\cap L:\forall y\in{}^{\lambda}2\cap L((\lambda,\in_x)\cong(\lambda,\in_y)\rightarrow x\leq_{L}y))^{L}\}.$$
First note that the relativization to $L$ of $(\lambda,\in_x)\cong(\lambda,\in_y)$ and $x\leq_{L}y$ are respectively $\Sigma_1(\lambda)$ and $\Delta_1$.  On the other hand, the formula $x\in \mathsf{WO}_{\lambda}\cap L$ is $\Sigma_1(\lambda)$. Indeed, $x\in \mathsf{WO}_{\lambda}\cap L$ if and only if $x\in \mathsf{WO}_{\lambda}$, which is, as already pointed out, $\Delta_1(\lambda)$, and $x$ is constructible, which is $\Sigma_1$ (see, e.g., \cite[Lemma 2.13]{Devlin}). We then write $x\in WO_{\lambda}\cap L$ as $\exists z_1\theta(x,z_1,\lambda)$ and $y\in L$ is $\Sigma_1$ as $\exists z_2\sigma(y,z_2)$. Then $A_{L}^V$ is defined by the formula:
\begin{equation*}
	\begin{split}
		x\in A_{L}^V \leftrightarrow& \,\forall y\forall z_2\forall w\exists x_{0}\exists z_1 \\&(\theta(x,z_1,\lambda)\wedge(\neg\sigma(y,z_2)\vee\neg\varphi(x,y,w,\lambda)\vee\psi(x,y,x_{0})))),
	\end{split}
\end{equation*}
which is $\Pi_2(\lambda)$.

Assume now that $0^{\sharp}$ doesn't exist, so that $L$ has the weak covering property and computes correctly the successor of every singular cardinal. Then, by an argument as in the proof of Proposition \ref{proposition:toy-example}, we get that $A_{L}^{V}$ doesn't have the $\lambda$-$\mathsf{PSP}$. We have proved the following:

\begin{corollary}\label{corollary:toy-example-if-0-sharp-pi2-set-without-psp} If $0^{\sharp}$ doesn't exist, there is a $\Pi_2(\lambda)$-definable subset of ${}^{\lambda}2$ without the $\lambda$-$\mathsf{PSP}$. 
\end{corollary} 

In other words, a sufficient condition for the existence of $0^{\sharp}$ is that all $\Pi_2(\lambda)$-definable subsets of ${}^{\lambda}2$ have the $\lambda$-$\mathsf{PSP}$. It is easy to see that the same holds for $a^{\sharp}$ for every $a\in{}^{\lambda}2$. Observe also that Proposition \ref{proposition:toy-example} and Corollary \ref{corollary:toy-example-Delta_1-definable-subset-wo-psp} hold when $\lambda=\omega$. The proof of Corollary \ref{corollary:toy-example-if-0-sharp-pi2-set-without-psp} breaks down for countable $\lambda$, as the Covering Lemma for $L$ applies only to cardinals $\lambda$ greater than $\omega$.


\section{The $\lambda\hyp\mathsf{PSP}$ up to the existence of an inner model with a measurable cardinal}\label{section:dodd-jensen}

In this section, we prove that if there is a strong limit cardinal $\lambda$ of countable cofinality such that all (lightface) $\lambda\hyp\Pi^1_1$ subsets of ${}^{\lambda}2$ have the $\lambda\hyp\mathsf{PSP}$, then there is an inner model with a measurable cardinal. For this, we work with the Dodd-Jensen core model\footnote{For other recent applications of the Dodd-Jensen core model, see \cite{ben2025definable}.}. For an intuitive picture of Dodd-Jensen premice and mice, the reader may look at the first paragraphs of the next section.

    \begin{definition}\label{definition:dj-premouse} Let $\kappa$ be an uncountable ordinal. $M$ is a \emph{Dodd-Jensen premouse at $\kappa$} if it is a structure $\mathcal{J}^{U}_{\alpha}$ of the form $( J_{\alpha}[U],\in,U)$ such that $$\mathcal{J}_{\alpha}^{U}\vDash``U \text{ is a normal measure on } \kappa".$$
    \end{definition}
    
    The reader, if not very familiar with Jensen's fine hierarchy, may prefer to think of $J_{\alpha}$ simply as $L_{\alpha}$ and will still be able to follow without difficulties.

    \begin{definition} Given $M$ a Dodd-Jensen premouse at $\kappa$, the \emph{lower part of $M$} is $lp(M)=M\cap V_{\kappa}$.
    \end{definition}

    Let $M$ be a Dodd-Jensen premouse at an uncountable ordinal $\kappa$ and let $\delta$ be an ordinal. A \emph{linear iteration of $M$ of length $\delta$} is a sequence $\langle M_{\alpha},\pi_{\alpha,\beta}:\alpha\leq\beta<\delta\rangle$ where
\begin{itemize}
	\item[(1)] $M_{0}=M$ and $\pi_{\alpha,\alpha}=id_{M_{\alpha}}$ for every $\alpha<\delta$,
	\item[(2)] If $\alpha+1<\delta$, $M_{\alpha}=\mathcal{J}^{U_{\alpha}}_{\eta_{\alpha}}$ is a Dodd-Jensen premouse at $\kappa_{\alpha}$ and $M_{\alpha+1}=\mathcal{J}^{U_{\alpha+1}}_{\eta_{\alpha+1}}$ is a Dodd-Jensen premouse at $\kappa_{\alpha+1}$, then $M_{\alpha+1}$ is the transitive collapse of the ultrapower of $M_{\alpha}$ by $U_{\alpha}$ with $\pi_{\alpha,\alpha+1}:M_{\alpha}\longrightarrow M_{\alpha+1}$ the corresponding $\Sigma_1$-elementary embedding, $\pi_{\alpha,\alpha+1}(\kappa_{\alpha})=\kappa_{\alpha+1}$ and 
	\[
	U_{\alpha+1}=\{[f]_{U_{\alpha}}:f\in  {}^{\kappa_{\alpha}}M_{\alpha}\cap M_{\alpha}\wedge \{\delta<\kappa_{\alpha}:f(\delta)\in U_{\alpha}\}\in U_{\alpha}\}.
	\]
	\item[(3)] If $\gamma<\delta$ is a limit ordinal, then $M_{\gamma}$ together with the $\Sigma_1$-elementary embeddings $\pi_{\alpha,\gamma}:M_{\alpha}\longrightarrow M_{\gamma}$ for every $\alpha<\gamma$ is the direct limit of the directed system $\langle M_{\alpha},\pi_{\alpha,\beta}:\alpha\leq\beta<\gamma\rangle$. 
\end{itemize}

To secure the existence of the transitive collapse of an ultrapower, it needs to be well-founded. If at step $\alpha$ of the iteration the resulting ultrapower isn't well-founded, the iteration halts and the Dodd-Jensen premouse doesn't have an iteration of length $\alpha$.

\begin{definition}\label{definition:dj-mouse} A Dodd-Jensen premouse is \emph{iterable} if it has a linear iteration of length $\delta$ for every ordinal $\delta$. $M$ is a \emph{Dodd-Jensen mouse} if it is an iterable Dodd-Jensen premouse.
\end{definition}

The iterability of Dodd-Jensen premice is witnessed by any transitive structure large enough to see the well-foundedness of their $\omega_1$-first iterated ultrapowers:

\begin{theorem}[{\cite[Theorem 2.7]{koepke1988some}}]\label{theorem:iterability-iff-iterability-in-a-large-enough-transitive-model} Let $M$ be a Dodd-Jensen premouse and let $H$ be a transitive model of a sufficiently large finite part of $ZFC$ such that $M\in H$ and $\omega_{1}^{V}\subseteq H$. Then $M$ is iterable if and only if $H\vDash``M \text{ is iterable}"$.
\end{theorem} 

Moreover, as the following shows, Dodd-Jensen mice can be compared via their corresponding iterates. As customary, if a Dodd-Jensen mouse $M$ is the initial segment of a Dodd-Jensen mouse $N$, we write $M\trianglelefteq N$.

\begin{theorem}[{\cite[Theorem 2.12]{koepke1988some}}]\label{theorem:comparison-of-dj-mice} For every two Dodd-Jensen mice $M,N$, there are iterates $M_{\delta}$ and $N_{\delta'}$ such that either $M_{\delta}\trianglelefteq N_{\delta'}$ or $N_{\delta'}\trianglelefteq M_{\delta}$. 
\end{theorem}

The lower part of a Dodd-Jensen mouse remains fixed throughout its iterations. Because any two Dodd-Jensen mice have iterates such that one is an initial segment of the other, the lower parts of the Dodd-Jensen mice are ordered by inclusion. This fact leads naturally to the definition of the Dodd-Jensen core model:

\begin{definition} The \emph{Dodd-Jensen core model} is the structure
	\[
	K^{\mathsf{DJ}}:=L\cup\bigcup\{lp(M):M\textit{ is a Dodd-Jensen mouse}\}.
	\]
\end{definition}

Recall that $0^\sharp$ can be seen as the first mouse (see, e.g., \cite[Lemma 35.7]{Jech}). Then, if $0^{\sharp}$ doesn't exist, the Dodd-Jensen core model is $L$; and if it does, then $L\subsetneq K^{\mathsf{DJ}}$. 

\begin{theorem}[Dodd-Jensen]\text{}\label{theorem:dodd-jensen-core-model-properties}
	\begin{enumerate-(1)-r}
	    \item $K^{\mathsf{DJ}}\vDash\mathsf{ZFC+GCH}$.\label{theorem:dodd-jensen-models-zfc-gch}
        \item $K^{\mathsf{DJ}}$ admits a $\Sigma_1$-definable well-order.\label{theorem:dodd-jensen-sigma-1-definable-well-order}
	\end{enumerate-(1)-r}
\end{theorem}

The well-ordering of the Dodd-Jensen core model is given by
\[
x\leq_{\mathsf{DJ}}y \leftrightarrow \exists M(M\textit{ is a Dodd-Jensen mouse}\wedge x,y\in lp(M) \wedge M\vDash``x\leq_{M}y"),
\]
where $\leq_{M}$ denotes the constructibility well-order of the mouse $M$, which is uniformly $\Sigma_{1}(M)$ (see \cite[p.186]{koepke1988some}).

\begin{definition}\label{def:covering-properties}
	An inner model $M$ has the \emph{full covering property} if for every uncountable set of ordinals $X$ in $V$ there is a $Y\in M$ such that $X\subseteq Y$ and $|Y|=|X|$. It has the \emph{weak covering property} if it computes correctly the successor of singular cardinals, i.e., if for every singular cardinal $\lambda$, $(\lambda^{+})^{M}=\lambda^{+}$.
\end{definition}

It is a well known fact that any inner model satisfying the axiom of choice and the full covering property also satisfies the weak covering property. The following gives the anti-large cardinal assumption under which the Dodd-Jensen core model exhibits these properties:

\begin{theorem}[Dodd-Jensen's Covering Theorem for $K^{\mathsf{DJ}}$]\label{cvthmkdj} If there is no inner model with a measurable cardinal, then $K^{\mathsf{DJ}}$ has the full and weak covering properties.
\end{theorem}

\begin{remark}
 It follows from Theorem \ref{theorem:dodd-jensen-core-model-properties}\ref{theorem:dodd-jensen-sigma-1-definable-well-order} and Proposition \ref{proposition:toy-example} that in $K^{DJ}$ there is a $\Pi_1(\lambda)$-subset without the $\lambda\hyp\mathsf{PSP}$. Because of Theorem \ref{cvthmkdj}, we have the following strengthening of Corollary \ref{corollary:toy-example-if-0-sharp-pi2-set-without-psp}: if there is no inner model with a measurable cardinal, then there is a $\Pi_2(\lambda)$-definable subset of ${}^{\lambda}2$ without the $\lambda\hyp\mathsf{PSP}$. To see this, let $A^{V}_{K^{\mathsf{DJ}}}$ be the set $$\{x\in \mathsf{WO}_{\lambda}\cap K^{\mathsf{DJ}}:\forall y\in K^{\mathsf{DJ}}((\lambda,\in_{x})\cong(\lambda,\in_{y})\rightarrow x\leq_{{\mathsf{DJ}}}y)^{K^{\mathsf{DJ}}}\},$$
and proceed exactly as the equivalent result for $L$ in the previous section. Here, one only has to prove that $x\in K^{\mathsf{DJ}}$ is $\Sigma_1$. But this is clear: $x\in K^{\mathsf{DJ}}$ if and only if there exists a Dodd-Jensen mouse $M$ such that $x\in lp(M)$.
\end{remark}

We now turn our attention to the descriptive version of this problem. As shown by Dimonte and Motto Ros, in $L$ it is possible to construct a $\lambda$-coanalytic set that doesn't have the $\lambda\hyp\mathsf{PSP}$. 
In this section, we show that the same occurs in the Dodd-Jensen core model. In fact, we prove a slightly stronger result: namely, that in $K^{\mathsf{DJ}}$ there exists a (lightface) $\lambda\hyp\Pi^1_1$ subset of ${}^{\lambda}2$ without the $\lambda\hyp\mathsf{PSP}$. We will need the following preliminary lemmas:

\begin{lemma}\label{32}
	For every $x,y\in{}^{\lambda}2\cap K^{\mathsf{DJ}}$ there exists a Dodd-Jensen mouse $M$ of size $\lambda$ with $x,y\in lp(M)$.
\end{lemma}
\begin{proof}
	Let $x,y\in{}^{\lambda}2\cap K^{\mathsf{DJ}}$. Since $x,y\in K^{\mathsf{DJ}}$, there exist two Dodd-Jensen mice $M_{x}$ and $M_{y}$ at $\kappa_x$ and $\kappa_y$ respectively such that $x\in lp(M_{x})$ and $y\in lp(M_{y})$. Let $M_{x}'$ be the collapse of $Hull^{M_{x}}(\lambda\cup\{x,\kappa_x\})$ and let $\pi_{x}^{-1}$ be the collapse function. By elementarity, $M'_x$ is a Dodd-Jensen mouse at $\pi_{x}^{-1}(\kappa_x)$ of size $\lambda$ and $x\in lp(M'_x)$. To see the latter, note that $\pi^{-1}_x(x)=x$ and $rank^{M'_x}(x)$ is still below $rank^{M'_x}(\pi_{x}^{-1}(\kappa_x))$. Analogously, we let $M'_{y}$ be the collapse of $Hull^{M_{y}}(\lambda\cup\{y,\kappa_y\})$ with $\pi_{y}^{-1}$ the corresponding collapse function, so that $M'_y$ is a Dodd-Jensen mouse at $\pi_{y}^{-1}(\kappa_y)$ of size $\lambda$ with $y\in lp(M'_y)$.  By Theorem \ref{theorem:comparison-of-dj-mice}, by iterating both $M'_{x}$ and $M'_{y}$ we get two Dodd-Jensen mice $\tilde{M}_{x}$ and $\tilde{M}_{y}$ such that, with no loss of generality, $\tilde{M}_{x}\trianglelefteq\tilde{M}_{y}$. The coiteration successfully removes all disagreements\footnote{This is inner model theory jargon. We simply mean that the iteration has finally produced two models such that one is the initial segment of the other.} in $\nu<\max(|M'_x|,|M'_y|)^{+}=\lambda^{+}$ steps (see, e.g., the proof of \cite[Lemma 4.8]{steel2016introduction}) and it also holds that $|\tilde{M}_i|=|M'_i||\nu|$ for $i\in\{x,y\}$. It follows that $|\tilde{M}_{y}|=\lambda$. Moreover, $x,y\in lp(\tilde{M}_{y})$ because the lower parts of $M'_x$ and $M'_y$ remain unchanged throughout the iteration. To finish, we let $\tilde{M}_y$ to be our desired $M$. 
\end{proof}

Let now $x,y\in{}^{\lambda}2\cap K^{\mathsf{DJ}}$ be such that $x\leq_{\mathsf{DJ}}y$ and recall that $x\leq_{\mathsf{DJ}}y$ if and only if there exists a Dodd-Jensen mouse $M$ such that $x,y\in lp(M)$ and $M\vDash ``x\leq_{M}y"$. Arguing as in the lemma above, the transitive collapse $M'$ of $Hull^{M}(\lambda\cup\{x,y\})$ is a Dodd-Jensen mouse of size $\lambda$ with both $x,y\in lp(M')$ such that $M'\vDash``x\leq_{M'}y"$. Note as well that the Dodd-Jensen mouse $M'$ contains $\lambda$. This proves the following:

\begin{corollary}\label{corollary:x-less-than-y-iff-less-in-dj-mouse} For every $x,y\in{}^{\lambda}2\cap K^{\mathsf{DJ}}$, $x\leq_{\mathsf{DJ}}y$ if and only if there exists a Dodd-Jensen mouse $M$ of size $\lambda$ with $\lambda\in M$ such that $x,y\in lp(M)$ and $M\vDash ``x\leq_{M}y"$. 
\end{corollary}

\begin{remark}
	It readily follows from the proof above that $x\leq_{DJ}y$ if and only if $M\vDash ``x\leq_{M}y"$ \emph{for all} Dodd-Jensen mice $M$ with $x,y\in lp(M)$. That is, $\leq_{\mathsf{DJ}}$ is actually $\Delta_1$.  
\end{remark}

Let $M$ be a Dodd-Jensen mouse of size $\lambda$ and let $\theta$ be a cardinal large enough for $M\in H_{\theta}$ and $H_{\theta}\vDash \mathsf{ZF^{-}}$ to hold. The collapse of $Hull^{H_{\theta}}(M\cup\{M\})$ is then a transitive $\mathsf{ZF^{-}}$-model of size $\lambda$ containing $M$ as an element. Therefore, every Dodd-Jensen mouse of size $\lambda$ belongs to a transitive $\mathsf{ZF^{-}}$-model $H$ of size $\lambda$. For our purposes, it might be convenient to have a model satisfying only a finite fragment of $\mathsf{ZF^{-}}$. An argument as above secures the existence of such a model. In the context of this chapter, the infinite fragment of $\mathsf{ZF^{-}}$ that we require the collapse of $Hull^{H_{\theta}}(M\cup\{M\})$ to satisfy is that needed to witness the iterability of its Dodd-Jensen mice elements. To sum up:

\begin{lemma}\label{l34b} For every Dodd-Jensen mouse $M$ of size $\lambda$ there is a transitive 
	structure $H$ of size $\lambda$ such that $M\in H$ and which models a finite fragment of $\mathsf{ZF^{-}}$ enough to witness the iterability of its Dodd-Jensen mice.
\end{lemma}

If $M$ and $H$ are as above, then $\omega_1^{V}\subseteq H$ because $\lambda$ is assumed to be uncountable, so, by Theorem \ref{theorem:iterability-iff-iterability-in-a-large-enough-transitive-model}, $M$ is iterable if and only if $H\vDash``M\text{ is iterable}"$.

\subsection{Coding Dodd-Jensen mice} Let $x,y\in{}^{\lambda}2\cap K^{\mathsf{DJ}}$. By Corollary \ref{corollary:x-less-than-y-iff-less-in-dj-mouse}, $x\leq_{\mathsf{DJ}}y$ if and only if there exists a Dodd-Jensen mouse $M$ of size $\lambda$ such that $x,y\in lp(M)$ and $M\vDash ``x\leq_{M}y"$. Equivalently, $x\leq_{\mathsf{DJ}}y$ if and only if there is a code $z\in{}^{\lambda}2$ for a well-founded and extensional structure $(\lambda,\in_{z})$ whose (bounded) transitive collapse $tr(\lambda,\in_{z})$ is a Dodd-Jensen mouse such that $x,y\in lp(tr(\lambda,\in_{z}))$ and $tr(\lambda,\in_{z})\vDash``x\leq_{tr(\lambda,\in_{z})}y"$. 

Note that a Dodd-Jensen premouse is a structure of an augmented language $\mathcal L_{\in}(\dot{U})$ with a unary predicate. Let $((z)_1,(z)_2)\in{}^{\lambda}2\times{}^{\lambda}2$ be such that $(z)_1$ is the code for a bounded, extensional and well-founded structure $(\lambda,\in_{(z)_1})$. By defining $U_{(z)_2}=\{\alpha<\lambda:(z)_2(\alpha)=1\}$, we get that $((z)_1,(z)_2)$ codes an $\mathcal{L}_{\in}(\dot{U})$-model $(\lambda,\in_{(z)_1},U_{(z)_2})$\footnote{It is possible to code a structure $(M,E,U)$ with a single element $z\in{}^{\lambda}2$ by letting $E=\{(\alpha,\beta):z(\langle \alpha,\beta\rangle)=1\}$ and $U=\{\beta:z(\beta)=0\}$, but this way things become unnecessarily cumbersome.}.

By $\mathsf{PM^{DJ}}(\lambda)$ we denote the set of pairs in $\mathsf{BC}_{\lambda}\times{}^{\lambda}2$ coding Dodd-Jensen premice of size $\lambda$, that is:
\begin{equation*}
	\begin{split}
		\mathsf{PM^{DJ}}(\lambda):=&\{((z)_1,(z)_2)\in \mathsf{BC}_{\lambda}\times{}^{\lambda}2:(\lambda,\in_{(z)_1},U_{(z)_2})\vDash``\exists\kappa(U_{(z)_2}\text{ is}\\&\text{a normal measure on } \kappa)\wedge V=L[U_{(z)_2}]"\},
	\end{split}
\end{equation*}
where the last condition together with $tr(\lambda,\in_{(z)_1},U_{(z)_2})$ being transitive makes $tr(\lambda,\in_{(z)_1},U_{(z)_2})$ a structure of the form $\mathcal{J}^{U_{(z)_2}}_{\alpha}$, as the following tells (structures of the form $(\lambda,\in_{(z)_1},U_{(z)_2})$ as defined above already satisfy that, for all $x\in\lambda$, $x\in U_{(z)_2}$ if and only if $(\lambda,\in_{(z)_1},U_{(z)_2})\vDash x\in U_{(z)_2}$):

\begin{lemma}[See, e.g., {\cite[Lemma 5.28]{schindler2014set}}]\label{l39} Let $M$ be a transitive model. Then, $M\vDash V=L[\dot{U}]$ if and only if $M=J_{\alpha}[\dot{U}]$ for some $\alpha$ and $\dot{U}$, where $x\in\dot{U}$ if and only if $M\vDash x\in\dot{U}$ for all $x\in M$. 
\end{lemma}

\begin{lemma}\label{SPM} $\mathsf{PM^{DJ}}(\lambda)$ is \emph{(lightface)} $\lambda\hyp\Pi^1_1$ in ${}^{\lambda}2\times{}^{\lambda}2$.
\end{lemma} 
\begin{proof} It follows from Proposition \ref{prop:BC-is-pi-11} and Lemma \ref{lemma:complexity_of_satisfaction} that the set $\{((z)_1,(z)_2)\in\mathsf{BC}_\lambda\times{}^{\lambda}2:(\lambda,\in_{(z)_1},U_{(z)_2})\vDash\varphi\}$
is $\lambda\hyp\Pi^1_1$ for every $\mathcal{L}(\dot{U})$-formula $\varphi$.
\end{proof}

We denote by $\mathsf{M^{DJ}}(\lambda)$ the set of codes for Dodd-Jensen mice. That is:
\begin{equation*}
	\begin{split}
		\mathsf{M^{DJ}}(\lambda):=&\{((z)_{1},((z)^1_2,(z)^2_2))\in\mathsf{BC}_{\lambda}\times\mathsf{PM^{DJ}}(\lambda):\\&\,((z)^1_2,(z)^2_2)\in tr(\lambda,\in_{(z)_1})\wedge tr(\lambda,\in_{(z)_{1}})\vDash\mathcal{T}\\&\wedge tr(\lambda,\in_{(z)_{1}})\vDash``tr(\lambda,\in_{(z)^1_2},U_{(z)^2_2})\text{ is iterable}"\},
	\end{split}
\end{equation*}
where $\mathcal{T}$ is a finite fragment of $\mathsf{ZF^{-}}$ large enough to witness the iterability of Dodd-Jensen premice. The existence of a structure such as the one coded by $(z)_{1}$ is secured by Lemma \ref{l34b}.

\begin{lemma}\label{PM}  $\mathsf{M^{DJ}}(\lambda)$ is \emph{(lightface)} $\lambda\hyp\Pi^1_1$ in ${}^\lambda 2\times{}^{\lambda}2$.
\end{lemma}
\begin{proof}
    By the previous, coding a Dodd-Jensen premouse is $\lambda\hyp\Pi^1_1$. The rest follows from Proposition \ref{prop:BC-is-pi-11} and Lemma \ref{lemma:complexity_of_satisfaction}, as above.
\end{proof}

A consequence of the previous two lemmas is the following:

\begin{proposition}\label{lemma:order-and-cantor-are-sigma-1-2} Both ${}^{\lambda}2\cap K^{\mathsf{DJ}}$ and $\leq_{\mathsf{DJ}}\upharpoonright({}^{\lambda}2\times{}^{\lambda}2)\cap K^{\mathsf{DJ}}$ are \emph{(lightface)} $\lambda\hyp\Sigma^1_2$. 
\end{proposition}
\begin{proof} We have already seen that $x\in{}^{\lambda}2\cap K^{\mathsf{DJ}}$ if and only if there is a Dodd-Jensen mouse $M$ of size $\lambda$ such that $x\in lp(M)$. That is, $x\in{}^{\lambda}2\cap K^{\mathsf{DJ}}$ if and only if $\exists z\in\mathsf{M^{\mathsf{DJ}}}(\lambda)(x\in lp(tr(\lambda,\in_{(z)^1_2}))$. Note that $x\in lp(tr(\lambda,\in_{(z)^1_2}))$ can be written as $x\in tr(\lambda,\in_{(z)^1_2})\wedge tr(\lambda,\in_{(z)^1_2})\vDash``x\in V_{\bigcup(z)^2_2}"$, where $(z)^{2}_{2}$ is the code of the measure as in Lemma \ref{SPM}. The first statement defines a $\lambda\hyp\Pi^1_1$ set by Lemma \ref{lemma:y-in-transitive-closure-of-lambda-inx-is-pi11}, while the second defines a $\lambda\hyp\Delta^1_1$ set by Lemma \ref{lemma:complexity_of_satisfaction}. By Lemma \ref{PM}, $z\in\mathsf{M^{DJ}}(\lambda)$ is $\lambda\hyp\Pi^1_1$, so $\exists z\in\mathsf{M^{DJ}}(\lambda)$ is $\lambda\hyp\Sigma^1_2$. Altogether, the sentence defines a (lightface) $\lambda\hyp\Sigma^1_2$ set. 

For the second statement, simply note that if $x,y\in{}^{\lambda}2\cap K^{\mathsf{DJ}}$, then $x\leq_{D}y$ if and only if $\exists z\in\mathsf{M^{DJ}}(\lambda)(x,y\in lp(tr(\lambda,\in_{(z)^1_{2}}))\wedge tr(\lambda,\in_{(z)^{1}_{2}})\vDash``x\leq y")$, where $\leq$ stands for the constructibility order of the mouse $tr(\lambda,\in_{(z)^{1}_{2}})$. As before, $x,y\in lp(tr(\lambda,\in_{(z)^{1}_{2}}))$ is $\lambda\hyp\Pi^1_1$. Since the order of every mice $M$ is uniformly $\Sigma_1(M)$, $tr(\lambda,\in_{(z)^{1}_{2}})\vDash``x\leq y"$ defines a $\lambda\hyp\Delta^1_1$ set by Lemma \ref{lemma:complexity_of_satisfaction}. The rest goes as before.
\end{proof}

\begin{remark}\label{remark:y-in-lower-part-of-x-is-pi-1-1}
    For later use, it is worth noting that $\{(x,y) : x \in lp(tr(\lambda,\in_y))\}$ is $\lambda\hyp\Pi^1_1$, as showed in the proof above.
\end{remark}

From now on, we write $x \sim y$ to denote $(\lambda,\in_x) \cong (\lambda,\in_y)$.
 The set $\{(x,y)\in{}^{\lambda}2\times{}^{\lambda}2:x\sim y\}$ is $\lambda\hyp\Sigma^1_1$, for $(\lambda,\in_x)\cong(\lambda,\in_y)$ if and only if there exists an order preserving bijection $f\in{}^{\lambda}\lambda$.
By the previous, the set $\leq_{\mathsf{DJ}}\upharpoonright({}^{\lambda}2\times{}^{\lambda}2)\cap K^{\mathsf{DJ}}$ is $\lambda\hyp\Sigma^{1}_{2}$. Therefore the set $$A:=\{x\in\mathsf{WO}_{\lambda}\cap{}^{\lambda}2:\forall y\in{}^{\lambda}2(y\leq_{\mathsf{DJ}}x\rightarrow 
y\nsim x))\},$$ is $\lambda\hyp\Pi^{1}_{2}$. Note that $A$ cannot have the $\lambda\hyp\mathsf{PSP}$ in $K^{\mathsf{DJ}}$: by its definition, for every $\alpha<\lambda^{+}$ there is a unique $x\in{}^{\lambda}2$ coding a well-order on $\lambda$ of order type $\alpha$, so $|A|=\lambda^{+}$ (for $\mathsf{GCH}$ holds in $K^{\mathsf{DJ}}$). Then, if $A$ had the $\lambda\hyp\mathsf{PSP}$, there would be an embedding $\pi$ of ${}^{\lambda}2$ into $A$ and the $\lambda$-analytic set $\pi[{}^{\lambda}2]$ would be of size $\lambda^{+}$. Consequently, $\{ot(<_{x}):x\in\pi[{}^{\lambda}2]\}$ would be an unbounded $\lambda$-analytic set in $\lambda^{+}$, contradicting the Boundedness Lemma. This already proves that there is a (lightface) $\lambda\hyp\Pi^1_2$ subset of ${}^{\lambda}2$ without the $\lambda\hyp\mathsf{PSP}$ in $K^{\mathsf{DJ}}$.

Because of how $\leq_{\mathsf{DJ}}$ is defined, if $x\in{}^{\lambda}2\cap K^{\mathsf{DJ}}$ is the $\leq_{\mathsf{DJ}}$-least element satisfying certain property $P$, then $x$ is the $\leq_{\mathsf{DJ}}$-least element in some Dodd-Jensen mouse $M$ with $x\in lp(M)$ such that $P(x)$. Conversely, if $x\in{}^{\lambda}2\cap K^{\mathsf{DJ}}$ is the $\leq_{M}$-least element in the lower part of some Dodd-Jensen mouse $M$ satisfying $P$, then it is the $\leq_{\mathsf{DJ}}$-least. Therefore, the set $A$ from the previous paragraph can be written instead as
\begin{equation*}
	\begin{split}
		\{x\in\mathsf{WO}_\lambda:&\exists z\in\mathsf{M^{DJ}}(\lambda)\forall y\in lp(tr(\lambda,\in_{z}))\cap{}^{\lambda}2\\&(tr(\lambda,\in_{z})\vDash``y<_{tr(\lambda,\in_{z})} x"\rightarrow y\nsim x)\},
	\end{split}
\end{equation*}
which is easily seen to be (lightface) $\lambda\hyp\Sigma^1_2$.

A trick of Solovay can be used to further lower the descriptive complexity of $A$. Let $z$ be the $\leq_{\mathsf{DJ}}$-least code for a structure $(\lambda,\in_{z})$ such that for all $y\in lp(tr(\lambda,\in_{z}))\cap{}^{\lambda}2$, if $tr(\lambda,\in_{z})\vDash``y\leq_{tr(\lambda,\in_{z})} x"$, then $y\nsim x$. To say that such $z$ is the $\leq_{\mathsf{DJ}}$-least is not $\lambda\hyp\Pi^1_1$, but it is $\lambda\hyp\Pi^1_1$ to say of some $z'\in\mathsf{PM^{DJ}}(\lambda)$ that $z'$ codes a Dodd-Jensen mouse $M_{z'}$ with $z\in lp(M_{z'})$ which witnesses that $z$ is indeed the $\leq_{\mathsf{DJ}}$-least. Again, to specify such $z'$ is not $\lambda\hyp\Pi^1_1$, so we pick a third code $z''\in\mathsf{PM^{DJ}}(\lambda)$ for a Dodd-Jensen mouse $M_{z''}$ with $z'\in lp(M_{z''})$ where $z'$ is indeed the $\leq_{\mathsf{DJ}}$-least, and so on, and to say that of such $z''$ is again $\lambda\hyp\Pi^1_1$. This way, we inductively produce a sequence $\langle z_{n}:n\in\omega\rangle$. But this sequence must be unique, so instead of \emph{there is a sequence $\langle z_{n}:n\in\omega\rangle$ such that...} we can say \emph{for all sequences of codes $\langle z_{n}:n\in\omega\rangle$ such that...}. This makes $A$ (lightface) $\lambda\hyp\Pi^1_1$. More formally:

\begin{theorem}\label{theorem:weak-covering-dj-implies-pi11-without-psp} If $(\lambda^{+})^{K^{\mathsf{DJ}}}=\lambda^{+}$, then there is a \emph{(lightface)} $\lambda\hyp\Pi^1_1$ subset of ${}^{\lambda}2$ without the $\lambda\hyp\mathsf{PSP}$. In particular, there is such a set in $K^{\mathsf{DJ}}$. 
\end{theorem}
\begin{proof}
 
    Let $A'$ be the set of $z\in{}^{\omega}({}^{\lambda}2)$ such that: 
    \begin{enumerate}
        \item $(z)_0$ is the code for an ordinal $\alpha<\lambda^{+}$,\label{theorem:weak-covering-dj-implies-pi11-without-psp-i}
        \item $((z)_1,((z)_2,(z)_3))\in\mathsf{M^{DJ}}(\lambda)$ and $(z)_0\in lp(tr(\lambda,\in_{(z)_{2}}))\cap {}^{\lambda}2$,\label{theorem:weak-covering-dj-implies-pi11-without-psp-ii}
        \item $(z)_{0}$ is the $\leq_{\mathsf{DJ}}$-least element of $lp(tr(\lambda,\in_{(z)_{2}}))\cap{}^{\lambda}2$ coding the ordinal $\alpha$ in \eqref{theorem:weak-covering-dj-implies-pi11-without-psp-i},\label{theorem:weak-covering-dj-implies-pi11-without-psp-iii}
        \item $((z)_{4},((z)_5,(z)_6))\in\mathsf{M^{DJ}}(\lambda)$ and $((z)_1,((z)_2,(z)_3))\in lp(tr(\lambda,\in_{(z)_{5}}))$,\label{theorem:weak-covering-dj-implies-pi11-without-psp-iv}
        \item $((z)_{1},((z)_2,(z)_3))$ is the $\leq_{\mathsf{DJ}}$-least in $lp(tr(\lambda,\in_{(z)_{5}}))$ satisfying \eqref{theorem:weak-covering-dj-implies-pi11-without-psp-iv}.\label{theorem:weak-covering-dj-implies-pi11-without-psp-v}
    \end{enumerate}
    And then, for every $n\in\omega$,
    \begin{enumerate}[resume]
        \item $((z)_{3n+4},((z)_{3n+5},(z)_{3n+6}))\in\mathsf{M^{DJ}}(\lambda)$,\label{theorem:weak-covering-dj-implies-pi11-without-psp-vi}
		\item $((z)_{3n+1},((z)_{3n+2},(z)_{3n+3}))\in lp(tr(\lambda,\in_{(z)_{3n+5}}))$, and\label{theorem:weak-covering-dj-implies-pi11-without-psp-vii}
		\item
		$((z)_{3n+1},((z)_{3n+2},(z)_{3n+3}))$ is the $\leq_{\mathsf{DJ}}$-least element in $lp(tr(\lambda,\in_{(z)_{3n+5}}))$ satisfying \eqref{theorem:weak-covering-dj-implies-pi11-without-psp-vii}.\label{theorem:weak-covering-dj-implies-pi11-without-psp-viii}
    \end{enumerate}

    Even if ${}^{\omega}({}^{\lambda}2)$ is a $\lambda$-Polish space (in fact, it is $\lambda$-Borel isomorphic to ${}^{\lambda}2$, see \cite[Theorem 5.3.10]{DMR}), we have not formally defined a lightface hierarchy on it. For this reason, we code the description above inside ${}^\lambda 2$.

    Let $z\in{}^{\omega}({}^{\lambda}2)$ and let $\varphi((z)_n)$ be a formula that contains, as a parameter, the $n$-th element of $z$ for some $n\in\omega$. Then we change $\varphi((z)_n)$ to the following $\varphi'(z)$: $\forall z'\in{}^\lambda 2 (\forall\alpha<\lambda (\exists\beta<\lambda\, \alpha=\omega\cdot\beta+n \rightarrow z'(\beta)=z(\alpha)) \rightarrow \varphi(z'))$. Note that if $\varphi$ is a $\Pi^1_1$-formula, then also $\varphi'$ is a $\Pi^1_1$-formula. Define $A$ as the set of $z\in {}^\lambda 2$ that satisfies the change of the eight formulas above.
    
    Clause \eqref{theorem:weak-covering-dj-implies-pi11-without-psp-i} is $\lambda\hyp\Pi^1_1$ by Proposition \ref{WO_is_lightface_pi_1_1}. By Lemma \ref{PM} and Remark \ref{remark:y-in-lower-part-of-x-is-pi-1-1}, clauses \eqref{theorem:weak-covering-dj-implies-pi11-without-psp-ii}, \eqref{theorem:weak-covering-dj-implies-pi11-without-psp-iv}, \eqref{theorem:weak-covering-dj-implies-pi11-without-psp-vi} and \eqref{theorem:weak-covering-dj-implies-pi11-without-psp-vii} are $\lambda\hyp\Pi^1_1$, too. Clause \eqref{theorem:weak-covering-dj-implies-pi11-without-psp-iii} holds if and only if $(y\in lp(tr(\lambda,\in_{(z)_{2}}))\wedge tr(\lambda,\in_{(z)_{2}})\vDash``y\leq_{tr(\lambda,\in_{z_{1}})}(z)_{0}"\rightarrow tr(\lambda,\in_y)\ncong tr(\lambda,\in_{(z)_0}))$, which is $\lambda\hyp\Pi^1_1$ by Lemma \ref{lemma:complexity_of_satisfaction} and Remark \ref{remark:y-in-lower-part-of-x-is-pi-1-1}, and the same goes for \eqref{theorem:weak-covering-dj-implies-pi11-without-psp-v} and \eqref{theorem:weak-covering-dj-implies-pi11-without-psp-viii}. Altogether, $A$ is (lightface) $\lambda\hyp\Pi^1_1$.

    As mentioned above, $A$ does not have the $\lambda\hyp\mathsf{PSP}$ in $K^{\mathsf{DJ}}$.
    If $(\lambda^{+})^{K^{\mathsf{DJ}}}=\lambda^{+}$, the set $A$ is of size $\lambda^{+}$ also in $V$, and doesn't have the $\lambda\hyp\mathsf{PSP}$ by an entirely analogous argument. By Lemma \ref{lemma:order-and-cantor-are-sigma-1-2}, the set $A$ is still (lightface) $\lambda\hyp\Pi^1_1$ in $V$.
\end{proof}	


If there is no inner model with a measurable cardinal, then by the Dodd-Jensen Covering Theorem, $K^{\mathsf{DJ}}$ correctly computes the successor of every uncountable singular cardinal. Thus, the following holds: 

\begin{corollary}\label{c310}
	If there is an uncountable strong limit cardinal $\lambda$ of countable cofinality such that all \emph{(lightface)} $\lambda\hyp\Pi^1_1$ subsets of ${}^{\lambda}2$ have the $\lambda\hyp\mathsf{PSP}$, then there is an inner model with a measurable cardinal.
\end{corollary}

In particular, since all (lightface) $\lambda\hyp\Pi^1_1$ sets are $\lambda$-coanalytic (or (boldface) $\lambda\hyp\boldsymbol{\Pi}^1_1$), one gets that if $(\lambda^+)^{K^{\mathsf{DJ}}}=\lambda^+$, then there is a $\lambda$-coanalytic subset of ${}^{\lambda}2$ without the $\lambda\hyp\mathsf{PSP}$. Equivalently, if there exists an uncountable strong limit cardinal $\lambda$ with $cf(\lambda)=\omega$ such that all $\lambda$-coanalytic subsets of ${}^{\lambda}2$ have the $\lambda\hyp\mathsf{PSP}$, then there is an inner model with a measurable cardinal.

\begin{remark}
The argument used to prove Corollary~\ref{c310} doesn't work in the classical case simply because the Covering Theorem for $K^{\mathsf{DJ}}$ doesn't apply to the case $\lambda=\omega$. It is well known, moreover, that the exact consistency strength of all coanalytic sets having the $\mathsf{PSP}$ is at the existence of an inaccessible cardinal. Furthermore, one cannot adapt the previous discussion to $\lambda=\omega$ and derive that, if $\omega_1^{K^{\mathsf{DJ}}} = \omega_1$, then there exists a (lightface) $\Pi^1_1$ set of reals without the $\mathsf{PSP}$. In this situation, the obstacle is that one cannot code Dodd-Jensen mice of countable size as we have done. The set of codes $\mathsf{PM^{DJ}}(\omega)$ for Dodd-Jensen premice of countable size is defined analogously as before, and is indeed $\Pi^1_1$ as a subset of ${}^{\omega}2 \times {}^{\omega}2$. However, the situation changes when considering codes for Dodd-Jensen mice of countable size. In defining the set of codes $\mathsf{M^{DJ}}(\lambda)$, we have relied on Theorem~\ref{theorem:iterability-iff-iterability-in-a-large-enough-transitive-model}, which states that a Dodd-Jensen premouse is iterable if and only if it is iterable in a transitive structure large enough to witness its $\omega_1$-iterability. Such a structure must contain $\omega_1$ and is therefore uncountable. Thus, in the classical case one is forced to consider iterability from $V$, which is $\Pi^1_2$ (see, e.g., the proof of \cite[Theorem 20.18]{Kanamori}).
\end{remark}

\section{The $\lambda\hyp\mathsf{PSP}$ up to the existence of $0^{\dagger}$}\label{section:0-dagger}

As already mentioned, it follows from Corollary \ref{c310} and Proposition \ref{p24} that if there is a strong limit cardinal $\lambda$ of countable cofinality such that all $\lambda$-coanalytic subsets of ${}^{\lambda}2$ have the $\lambda\hyp\mathsf{PSP}$, then there is an inner model with a measurable cardinal. In this section, we lift this consistency strength lower bound to the existence of $0^\dagger$. 

If $\kappa$ is an ordinal which is measurable in some inner model we say that it is an \emph{internally measurable cardinal}. 
Let $\kappa$ be such an ordinal and let $U$ be a measure over $\kappa$ in $L[U]$. Recall that, as $L$, $L[U]$ is a model of $\mathsf{ZFC+GCH}$ which admits a definable well-ordering, which we denote by $\leq_{L[U]}$, where $x\leq_{L[U]}y$ if and only if there exists some limit ordinal $\delta$ such that $x,y\in L_{\delta}[U]$ and $L_{\delta}[U]$ satisfies the $\mathcal{L}_{\in}(\dot{U})$-sentence $``x\leq_{L[U]}y"$.
The model $L[U]$ lacks of condensation below its measurable, though. That is, while every elementary substructure of $L_\tau[U]$ with $\tau\geq\kappa$ is an initial segment of $L[U]$, this is no longer true if $M$ is an elementary substructure of $L_{\tau}[U]$ for some ordinal $\tau<\kappa$. However, in this case $M$ is, by elementarity, of the form $L_{\alpha}[U_{\alpha}]$, where $$L_{\alpha}[U_{\alpha}]\vDash``U_{\alpha}\text{ is a normal ultrafilter}".$$ Although $L_{\alpha}[U_{\alpha}]$ is not necessarily an initial segment of $L[U]$,  if $L_{\alpha}[U_{\alpha}]$ satisfies enough of $\mathsf{ZF^{-}}$, its iterated ultrapower $Ult^{(\kappa)}(L_{\alpha}[U_{\alpha}],U_{\alpha})$ is of the form $L_{\alpha'}[i_{\kappa}(U_{\alpha})]$ for some $\alpha'<\kappa^{+}$, where $i_{\kappa}$ is the corresponding elementary embedding $i_{\kappa}:L_{\alpha}[U_{\alpha}]\to Ult^{(\kappa)}(L_{\alpha}[U_{\alpha}],U_{\alpha})$. And because $i_{\kappa}(U_{\alpha})$ is generated by the closed and unbounded set of closed and unbounded sets of $\kappa$\footnote{For more details, check, e.g., \cite[Theorem 1.11]{mitchell2009beginning}.}, which belongs to $U$, it follows that $i_{\kappa}(U_{\alpha})\subseteq U$. Thus, the iterated ultrapower of $M$  \emph{is} an initial segment of $L[U]$. 

In particular, the Dodd-Jensen core model and $L[U]$ agree below $\kappa$. Together with the arguments from the previous section, this implies the following:

\begin{proposition}\label{lemma:if_kappa_below_lambda_no_coanalytic_with_the_psp}
	Assume $\lambda<\kappa$. If $(\lambda^+)^{L[U]}=\lambda^+$, then there is a \emph{(lightface)} $\lambda\hyp\Pi^1_1$ subset of ${}^{\lambda}2$ without the $\lambda\hyp\mathsf{PSP}$. In particular, there is such a set in $L[U]$. 
\end{proposition}
\begin{proof}
	By the assumption on the heigth of $\lambda$ and since the Dodd-Jensen core model and $L[U]$ agree up to $\kappa$, the (lightface) $\lambda\hyp\Pi^1_1$ subset $A$ of ${}^{\lambda}2$ without the $\lambda\hyp\mathsf{PSP}$ produced in $K^{\mathsf{DJ}}$ in the previous section  is in $L[U]$ as well, for its rank is below $\kappa$. 
    If $L[U]$ computes $\lambda^+$ correctly, the set $A$  is a (lightface) $\lambda\hyp\Pi^1_1$ subset of ${}^{\lambda}2$ without the $\lambda\hyp\mathsf{PSP}$ also in $V$. 
\end{proof}

Because of Prikry's Theorem (see, e.g., \cite[Lemma 10.6]{schindler2014set}), there is always a generic extension of $L[U]$ where $\kappa$ is still a cardinal, although of countable cofinality. That is, $L[U]$ cannot have a Covering Theorem as the one of $L$ or $K^{\mathsf{DJ}}$, and one cannot infer a result as in the previous section just from Proposition \ref{lemma:if_kappa_below_lambda_no_coanalytic_with_the_psp}. However, Dodd and Jensen showed that to get a Covering Theorem for $L[U]$ just the existence of such kind of forcing extensions is to be taken into account:

\begin{lemma}[Dodd-Jensen. See, e.g., {\cite[Theorem 35.16]{Jech}}]\label{theorem:dodd-jensen-covering-theorem-for-lu} Assume that there is an inner model with a measurable cardinal, let $\kappa$ be the least such cardinal and let $U$ be a measure on $\kappa$ in $L[U]$. Then:
	\begin{itemize}
		\item[\emph{(1)}] either $0^{\dagger}$ exists, or
		\item[\emph{(2)}] the Covering Theorem holds for $L[U]$, or
		\item[\emph{(3)}] there exists an $\omega$-sequence $\langle\kappa_{n}:n\in\omega\rangle$ Prikry generic over $L[U]$ 
		such that the Covering Theorem holds for $L[U][\langle\kappa_{n}:n\in\omega\rangle]$. 
	\end{itemize}
	Moreover, $L[U][\langle \kappa_n:n\in\omega\rangle]=L[\langle\kappa_n:n\in\omega\rangle]$.
\end{lemma}

\begin{remark}The fact that $L[U][\langle \kappa_n:n\in\omega\rangle]=L[\langle\kappa_n:n\in\omega\rangle]$ is, according to the comment right after Theorem 35.16 in \cite{Jech}, a consequence of Mathias' Theorem, but no proof is given. We\footnote{The authors thank Andreas Lietz for his help here.} give here a justification for this fact. Since the $\supseteq$ is evident, we only have to show that $L[U][\langle\kappa_n:n\in\omega\rangle]\subseteq L[U]$. To do this, let $\kappa$ be the measurable in $L[U]$ and let 
\[
U'=\{A\subseteq\kappa:\{\kappa_n:n\in\omega\}\setminus A\text{ is finite}\}\cap L[\langle\kappa_n:n\in\omega\rangle].
\] 
Note that for every ordinal $\alpha$, $L_{\alpha}[U]=L_{\alpha}[U']$. This can be proved by transfinite induction. The first step is trivial and so is every limit case, so suppose $L_\alpha[U]=L_{\alpha}[U']$. By a density argument, we have that \[U=\{A\subseteq\kappa:\{\kappa_n:n\in\omega\}\setminus A\text{ is finite}\}\cap L[U].\] Indeed, if $x\in U$ and $\mathbb{P}$ stands for the Prikry forcing, then the set $D_{x}:=\{(t,y)\in\mathbb{P}:y\subseteq x\}$ is dense, so if $(t,y)\in G\cap D_{x}$ where $G$ is Prikry generic and $t$ is $\kappa_0{}^{\smallfrown}\ldots{}^{\smallfrown}\kappa_m$ for some $m\in\omega$, then $\kappa_k\in y\subseteq x$ for every $k>m$, thus $x$ belongs to the intersection of $\{A\subseteq\kappa:\{\kappa_n:n\in\omega\}\setminus A\text{ is finite}\}$ and $L[U]$; and a similar argument shows the opposite direction. 

Now, since by induction hypothesis $L_{\alpha}[U]\in L[\langle\kappa_n:n\in\omega\rangle]$, it follows by the definition of $U'$ that 
\[
U\cap L_{\alpha}[U]=\{A\subseteq\kappa:\{\kappa_n:n\in\omega\}\setminus A\text{ is finite}\}\cap L_{\alpha}[U]=U'\cap L_{\alpha}[U].
\]
Finally, \[
L_{\alpha+1}[U]=\mathsf{Def}(L_{\alpha}[U],\in,U\cap L_{\alpha}[U])=\mathsf{Def}(L_{\alpha}[U'],\in,U'\cap L_{\alpha}[U])=L_{\alpha+1}[U'].
\]
In particular, $L[U]=L[U']\subseteq L[\langle\kappa_n:n\in\omega\rangle]$, as $U'\in L[\langle\kappa_n:n\in\omega\rangle]$.
\end{remark}

Now, since Prikry forcing does not change the successors of singular cardinals, Dodd-Jensen's theorem implies that if $0^\dagger$ does not exist, then $L[U]$ has the weak covering property. That is:

\begin{corollary}\label{corollary:all-lightface-pi-1-1-then-0-dagger}
	If there is a strong limit cardinal $\lambda<\kappa$ of countable cofinality such that all (lightface) $\lambda\hyp\Pi^1_1$ subsets of ${}^{\lambda}2$ have the $\lambda\hyp\mathsf{PSP}$, then $0^{\dagger}$ exists. 
\end{corollary}

Suppose now that $\lambda>\kappa$. Let $x\in{}^{\lambda}2\cap L[U]$ and pick $\tau$ a limit ordinal such that $x,U\in L_\tau[U]$. The transitive closure of $Hull^{L_{\tau}[U]}(\lambda\cup\{x,U\})$ is, by elementarity, of the form $L_\alpha[U]$ for some $\alpha$, has size $\lambda$, and contains both $x$ and $U$. Then, $x\in{}^{\lambda}2\cap L[U]$ if and only if $x\in L_\alpha[U]$ for some $\lambda\leq\alpha<\lambda^+$. Observe that since $\lambda$ is a strong limit cardinal, the measure $U$ can be coded by an element in ${}^{\lambda}2$. This may not be the case if $\lambda=\kappa$, for $U$ lives in $\mathcal{P}(\mathcal{P}(\lambda))$. The case  $\lambda=\kappa$ will be addressed later on. 

Similarly as above and still under the assumption that $\lambda>\kappa$, one can prove that $x,y\in{}^{\lambda}2$ are such that $x\leq_{L[U]}y$ if and only if there exists an ordinal $\alpha\in[\lambda,\lambda^+)$ such that $x,y\in L_\alpha[U]$ and $L[U]\vDash ``x\leq_{L_\alpha[U]}y"$.

Denote by $\mathsf{C}_{L[U]}(\lambda,\kappa)$ the set of codes for such structures. That is:
\begin{equation*}
	\begin{split}
		\mathsf{C}_{L[U]}(\lambda,\kappa):=&\{((z)_1,(z)_2)\in\mathsf{BC}_{\lambda}\times{}^{\lambda}2: (\lambda,\in_{(z)_1},U_{(z)_2})\vDash``V=L[U_{(z)_2}]\wedge\\ & U_{(z)_2}\text{ is a measure on }\kappa"\}.
	\end{split}
\end{equation*}

It is easy to see that:

\begin{lemma} $\mathsf{C}_{L[U]}(\lambda,\kappa)$ is \emph{(lightface)} $\lambda\hyp\Pi^1_1(\kappa)$ in ${}^\lambda 2\times{}^\lambda 2$. 
\end{lemma}

Thus, assuming $\lambda>\kappa$, the argument from the previous section can be carried out entirely in $L[U]$ to construct a (lightface) $\lambda\hyp\Pi^1_1(\kappa)$ subset of ${}^{\lambda}2$ without the $\lambda\hyp\mathsf{PSP}$ (in particular, one can also prove, provided $\lambda>\kappa$, that both ${}^{\lambda}2\cap L[U]$ and $\leq_{L[U]}\upharpoonright({}^{\lambda}2\times{}^{\lambda}2)$ are $\lambda$-coanalytic, or, more specifically, (lightface) $\lambda\hyp\Sigma^1_2(\kappa)$). This implies that, when $\lambda>\kappa$, if $(\lambda^+)^{L[U]}=\lambda^+$, then there is a (lightface) $\lambda\hyp\Pi^1_1(\kappa)$ subset of ${}^{\lambda}2$ without the $\lambda\hyp\mathsf{PSP}$. Together with Proposition \ref{lemma:if_kappa_below_lambda_no_coanalytic_with_the_psp} and Proposition \ref{p24}, this yields the following:

\begin{proposition}\label{propoision:no-pi-1-1-in-lu} Assume $\lambda>\kappa$.
	If $(\lambda^{+})^{L[U]}=\lambda^{+}$, then there is a \emph{(lightface)} $\lambda\hyp\Pi^1_1(\kappa)$ 
    subset of ${}^{\lambda}2$ without the $\lambda\hyp\mathsf{PSP}$. In particular, there is such a set in $L[U]$.
\end{proposition}

Assume now that $\lambda=\kappa$. Note that, in this case, $L[U]$ cannot have the covering property. Now, since $\lambda$ is of countable cofinality, there exists a countable sequence $\bar{\lambda}:=\langle\lambda_n:n\in\omega\rangle$ cofinal in $\lambda$. We can assume with no loss of generality that such a sequence is Prikry on $\lambda$. If such a sequence didn't exist, $L[U]$ would have no Prikry generic extensions. Let thus $M$ be the Prikry generic extension $L[U][\bar{\lambda}]$ over $L[U]$. By Lemma \ref{theorem:dodd-jensen-covering-theorem-for-lu}, $M=L[\bar{\lambda}]$. Given $z\in{}^{\lambda}2$ a code for $\bar{\lambda}$, and proceeding as in \cite[Theorem 7.2.12]{DMR}, we can find a $\lambda$-coanalytic subset $A$ of ${}^{\lambda}2$ in $L[z]$ without the $\lambda\hyp\mathsf{PSP}$. By the previous arguments, we can conclude that, under the assumption that all $\lambda$-coanalytic subsets of ${}^{\lambda}2$ have the $\lambda\hyp\mathsf{PSP}$, $L[U][\bar{\lambda}]$ doesn't have the Covering Property. 

In summary, if $\lambda$ is an arbitrary strong limit cardinal of countable cofinality such that every $\lambda$-coanalytic subset of ${}^{\lambda}2$ has the $\lambda\hyp\mathsf{PSP}$, then neither $L[U]$ nor any Prikry extension satisfies the covering property. The Dodd--Jensen covering theorem for $L[U]$ therefore yields, together with Proposition \ref{propoision:no-pi-1-1-in-lu} and Corollary \ref{corollary:all-lightface-pi-1-1-then-0-dagger}:

\begin{theorem}\label{theorem:conditions-0-dagger-exists}
	    If there is a strong limit cardinal $\lambda$ of countable cofinality such that either 
        \begin{enumerate-(1)-r}
            \item  $\lambda<\kappa$ and all (lightface) $\lambda\hyp\Pi^1_1$ subsets of ${}^{\lambda}2$ have the $\lambda\hyp\mathsf{PSP}$, or
            \item  $\lambda>\kappa$ and all (lightface) $\lambda\hyp\Pi^1_1(\kappa)$ subsets of ${}^{\lambda}2$ have the $\lambda\hyp\mathsf{PSP}$, or
            \item  $\lambda=\kappa$ and all $\lambda$-coanalytic subsets of ${}^{\lambda}2$ have the $\lambda\hyp\mathsf{PSP}$,
        \end{enumerate-(1)-r} 
        then $0^{\dagger}$ exists. 
\end{theorem}

From which it follows: 

\begin{corollary}\label{theorem:main-theorem}
	If there is an uncountable strong limit cardinal $\lambda$ of countable cofinality such that all $\lambda$-coanalytic subsets of ${}^{\lambda}2$ have the $\lambda\hyp\mathsf{PSP}$, then $0^{\dagger}$ exists.
\end{corollary}

\section{Final remarks}\label{section:final-remarks}

Corollary \ref{c310} and Theorem \ref{theorem:conditions-0-dagger-exists} are further steps of the analysis conducted in \cite{lucke2023definability}. We summarize the three results here:
\begin{itemize}
    \item Corollary \ref{c310}: ``all \emph{(lightface)} $\lambda\hyp\Pi^1_1$ sets have the $\lambda\hyp\mathsf{PSP}$'' implies the consistency of a measurable cardinal.
    \item Theorem \ref{theorem:conditions-0-dagger-exists}(1,2): ``if $\lambda$ is not the least internally measurable cardinal, for all $a\in H_\lambda$ all \emph{(lightface)} $\lambda\hyp\Pi^1_1(a)$  subsets of ${}^{\lambda}2$ have the $\lambda\hyp\mathsf{PSP}$'' implies the existence of  $0^{\dagger}$.
    \item \cite{lucke2023definability}[Theorems 1.1 and 1.2]: ``for all $a\in H_\lambda\cup\lambda$ all \emph{(lightface)} $\lambda\hyp\Sigma^1_2(a)$  subsets of ${}^{\lambda}2$ have the $\lambda\hyp\mathsf{PSP}$''  is equiconsistent to the existence of $\omega$-many measurable cardinals.
\end{itemize}

It is still open whether the consistency lower bounds we identified in the first two cases are optimal, and whether there is a consistency upper bound to all $\lambda$-coanalytics sets having the $\lambda\hyp\mathsf{PSP}$ that is lower than all the $\lambda$-projective sets having the $\lambda\hyp\mathsf{PSP}$.

Linear iterability seems to be important for the descriptive complexity of the set of codes for mice to remain at the level of (lightface) $\lambda\hyp\Pi^1_1$ and (boldface) $\lambda\hyp\boldsymbol{\Pi}^1_1$ subsets of ${}^{\lambda}2$. This suggests that our method (same selector, sets of codes, etc.) should still work below the level of $\zerohandgrenade$ (zero hand-grenade), where \emph{almost linear iterations} occur, although with certain restrictions: it may no longer be possible to apply this method to (parameter-free) (lightface) $\lambda\hyp\Pi^1_1$ subsets of ${}^{\lambda}2$, as it has already happened at the level of $0^\dagger$. Nevertheless, one can reasonably expect that the method still applies to $\lambda$-coanalytic subsets of ${}^{\lambda}2$.

	Inner models beyond $\zerohandgrenade$ give up on linear iterability in favour of iteration trees. Considerations regarding the existence of unique cofinal branches may imply a rise in the descriptive complexity of the codes for mice. One might consider generalising Steel’s work in \cite{steel1995projectively} to optimise iterability complexity in the presence of finitely many Woodin cardinals. But this introduces new challenges: Steel eliminates extra quantifiers using the Gandy-Spector Theorem, but we lack of an equivalent result in the context of $\mathsf{GDST}$ at singular cardinals of countable cofinality, where Determinacy fails already at the very first levels of the $\lambda$-Borel hierarchy (see \cite[\textsection 4.5]{DMR}). One could still carry out the same work renouncing to remain at the lowest levels in the $\lambda$-analytical or $\lambda$-projective hierarchies, but this introduces another difficulty: inner models $M_n$ with $n$-many Woodin cardinals do not have nice covering properties. Other alternatives still at the level of core models with Woodin cardinals present additional problems beyond those mentioned. For instance, the Jensen-Steel core model without a measurable cardinal lacks the structural features, such as the well-behaved building blocks found in the Dodd-Jensen or in $L[U]$, that our techniques crucially rely on. As a result, the arguments we've employed in those settings may no longer apply in this context.

    \bibliographystyle{alpha}
    \bibliography{bibliography}

    \end{document}